\documentclass{amsart}

\usepackage[english]{babel}
\usepackage[normalem]{ulem}
\usepackage[margin=1in]{geometry} 

\usepackage{amsmath, mathtools, amsthm, amssymb, soul}
\usepackage{graphicx}
\usepackage[colorlinks=true, allcolors=blue]{hyperref}
\usepackage[dvipsnames]{xcolor}
\usepackage{tikz}

\usetikzlibrary{arrows.meta, positioning,calc}
\usepackage{comment} 

\theoremstyle{definition} 
\newtheorem{theorem}{{Theorem}}[section]
\newtheorem{lemma}[theorem]{{Lemma}}
\newtheorem{proposition}[theorem]{{Proposition}}
\newtheorem{definition}[theorem]{{Definition}}
\newtheorem{notation}[theorem]{{Notation}}
\newtheorem{corollary}[theorem]{{Corollary}}
\newtheorem{example}[theorem]{{Example}}
\newtheorem{conjecture}[theorem]{{Conjecture}}
\newtheorem{remark}[theorem]{{Remark}}

\newtheorem{question}[theorem]{{Question}}
\newtheorem{problem}[theorem]{{Problem}}

\usepackage{todonotes}
\usepackage{standalone}
\usepackage{xcolor} 

\newcommand{\inv}{^{-1}}

\newcommand{\wt}{\widetilde}
\newcommand{\textmaxsl}{max-$sl$ } 

\newcommand{\R}{\mathbb{R}}

\newcommand{\N}{\mathbb{N}} 
\newcommand{\Tcal}{\mathcal{T}} 

\DeclareMathOperator{\interior}{int} 
\newcommand{\Crit}{\mathrm{Crit}} 

\makeatletter
\def\namedlabel#1#2{
    \begingroup
    \def\@currentlabel{#2}%
    \phantomsection\label{#1}
    \endgroup
}
\makeatother

\title{Decompositions and diagrams of symplectic surfaces in Weinstein domains}

\author[R.~Aranda]{Rom\'an Aranda}
\address{Department of Mathematics \\ Francis Marion University \\ 4822 E. Palmetto St, Florence, SC 29502}
\email{\href{mailto:romanaranda123@gmail.com}{romanaranda123@gmail.com}}

\author[P.~Cahn]{Patricia Cahn}
\address{Department of Mathematics \\ Smith College\\46 College Lane, Northampton, MA 01060}
\email{\href{mailto:pcahn@smith.edu}{pcahn@smith.edu}}

\author[A.~Roy]{Agniva Roy}
\address{Department of Mathematics \\ 
Boston College \\
Chestnut Hill, MA 02467}
\email{\href{mailto:agniva.roy@bc.edu}{agniva.roy@bc.edu}}

\author[M.~Zhang]{Melissa Zhang}
\address{Department of Mathematics, University of California, Davis, One Shields Avenue, Davis, CA 95616
}
\email{\href{mailto:mlzhang@ucdavis.edu}{mlzhang@ucdavis.edu}}

\date{August 28, 2026}

\begin{document}

\begin{abstract}
We introduce combinatorial and diagrammatic methods for representing properly embedded symplectic surfaces in 4-dimensional Weinstein domains.
We show that positive ascending surfaces, which include complex curves in Stein domains and multisections of Lefschetz fibrations, can be placed in  bridge position with respect to Islambouli--Starkston's bisection-with-divides structure on the Weinstein domain.
We algorithmically relate various decompositions of such surfaces, including transverse banded unlink diagrams, quasipositive factorizations, bridge bisections with divides, shadow diagrams (curves on surfaces), and pointed monodromy factorizations. 

We also develop a new way to present branched covers of Weinstein domains along positive ascending surfaces, which, combined with work of Loi--Piergallini, recovers Islambouli--Starkston's result that every compact Weinstein domain admits a bisection with divides.
\end{abstract}

\maketitle
\setcounter{tocdepth}{1}
\tableofcontents

\section{Introduction}
Over the past century \cite{Artin}, many diagrammatic descriptions for 2-knots, or surfaces, in 4-space have appeared~\cite{CS:Book}. In this work, we introduce combinatorial (i.e., diagrammatic, with finite data) 
techniques to represent symplectic surfaces properly embedded in 4-dimensional Weinstein domains. 
Seminal work of Rudolph \cite{rudolph1983algebraic} and Boileau--Orevkov \cite{boileau-orevkov} connected algebraic properties of braid closures in $S^3$ and the geometry of surfaces in the symplectic 4-ball, the simplest Weinstein domain.
Rudolph showed that the set of quasipositive braid\footnote{A braid that can be written as a product of conjugates of the standard Artin braid group generators.}  closures
in $S^3$ is equal to the set of transverse intersections of smooth complex curves with the unit sphere in $\mathbb{C}^2$. 
In particular, properly embedded complex surfaces in $B^4$ can be described by braided unlinks, together with half-twisted bands as in Figure~\ref{fig:fig_qp_trefoil}(left, middle). 
Hayden generalized this, showing
that positive ascending surfaces in Weinstein domains can be described by quasipositive pointed mapping classes \cite{hayden21}.

\begin{figure}[h]
    \centering
    \includegraphics[width=4.5in]{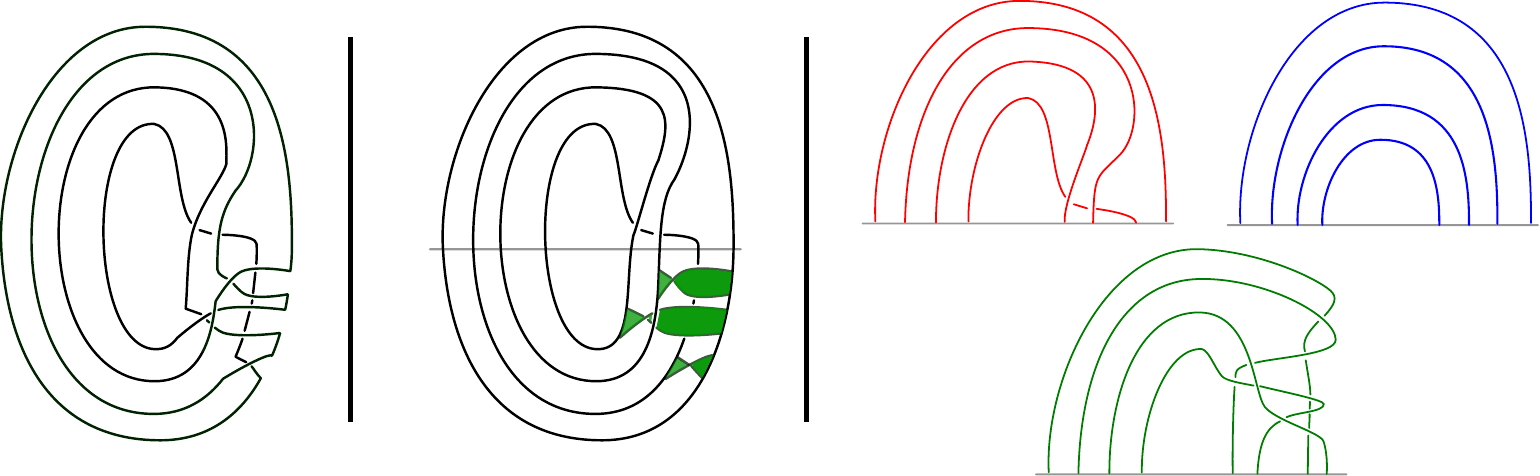}
    \caption{(Left) A quasipositive factorization of the trefoil knot.  (Middle) Banded diagram for a quasipositive surface bounded by the trefoil. (Right) Braided tri-plane diagram representing the same surface.}
    \label{fig:fig_qp_trefoil}
\end{figure}

In this work, we prove that positive ascending surfaces in Weinstein domains can be represented by systems of curves on surfaces, and equivalently, by certain factorizations of pointed mapping classes; see Tables~\ref{decompositions.tab} and~\ref{surfacedecompositions.tab}.  To achieve this, we prove the following theorem, which allows us to put such surfaces in bridge position with respect to a bisection-with-divides structure on the Weinstein domain, introduced by \cite{IS-divides}.

\begin{theorem}\label{thm: main}
    Let $S$ be a positive ascending surface in a Weinstein domain equipped with the structure of a bisection with divides.  Then $S$ can be isotoped through positive ascending surfaces (with the isotopies restricting to transverse isotopies of the boundary)  to a surface $S'$ which is in bridge position with respect to the bisection with divides.
\end{theorem}

We introduce a second diagram for symplectic surfaces in the standard $B^4$ called \emph{braided tri-plane diagrams} similar to those in \cite{aranda25}. 
These diagrams are triplets of braided tangles satisfying certain compatibility conditions as in Figure~\ref{fig:fig_qp_trefoil}(right). 
In this case, Theorem \ref{thm: main} shows that every smooth complex curve intersected with the unit sphere in $\mathbb{C}^2$ can be described combinatorially by a braided tri-plane diagram (Corollary \ref{cor: intro_complex}).

\subsection{Scientific context} Trisections were introduced by Gay and Kirby \cite{gay2016trisecting} as a generalization of Heegaard splittings to 4-manifolds. These provide a combinatorial way to encode smooth structures by splitting them into 1-handlebodies and can be considered an alternative perspective on handle decompositions. In particular, trisections allow the smooth structure to be encoded as a collection of multicurves on a  surface. These were generalized to multisections by Islambouli--Naylor \cite{Islambouli-Naylor}. 

Compatibility between symplectic structures and handle decompositions of smooth manifolds was first studied by Weinstein \cite{weinstein_handles}, and has since been explored in the works of Eliashberg\cite{Eliashberg90a}, Gompf \cite{gompf1998handlebody}, Gay \cite{gay_transverse}, and Etnyre--Min--Piccirillo--Roy \cite{empr}, with the latter two works focusing on symplectic 4-manifolds. 
Lambert-Cole--Meier--Starkston \cite{lambert2021symplectic} used branched covering techniques coming from work of Auroux \cite{auroux2000_symplectic} and Auroux--Katzarkov \cite{AK00_branched} to make progress towards characterizing trisections for closed symplectic 4-manifolds.
However, this perspective does not directly encode a natural handle decomposition.
In fact, it is in general not known whether symplectic 4-manifolds admit handle decompositions that are compatible with the symplectic structure in any meaningful way. For instance, the Weinstein trisections of \cite{lambert2021symplectic} do not encode the symplectic structure in a combinatorial diagram. 

Recently, Islambouli--Starkston \cite{IS-divides} defined  {\em multisections with divides} for Weinstein domains, a particular kind of symplectic manifold equipped with a Morse function admitting a gradient-like vector field that is the Liouville flow for the symplectic form (precise definition in Section~\ref{sec: background}). 
The multisection-with-divides decomposition is compatible with a Weinstein handle decomposition. 
The Weinstein structure is encoded as a collection of multicurves on a closed surface, which includes both the data of the corresponding smooth multisection, as well as the {\em dividing set} on a convex representative of the surface.

Meier--Zupan \cite{meier2017bridge, meier2018bridge} introduced the notion of bridge positions for smooth surfaces inside 4-manifolds, generalizing techniques used to study knots in 3-manifolds. Since then, a number of papers \cite{aranda2024bridge, PLC_surface, IKLC, meier_braided} have studied how to place smooth (or symplectic) surfaces in bridge position with respect to a trisection-like decomposition of the 4-manifold. 
This not only gives a combinatorial way to encode the data of such embedded submanifolds via systems of curves on surfaces, but also characterizes such submanifolds: the combinatorial decomposition can be used to reconstruct a smooth (or symplectic) surface. 
This has led to a fruitful area of research in using bridge trisections to compute invariants of knotted surfaces in 4-manifolds \cite{cahn23algorithms,BCTT22_Distance, ST_keicolourings, APT_kirbythompson24, APZ_kirbythompson23, aranda23pants}. Bridge trisections have further been applied to obtain diagrammatic descriptions of surgery operations and exotic 4-manifolds \cite{KM_pricetwist, naylor_twist}. It is established \cite{JMM_knottedsurface} that the theory of bridge trisections recovers several facets of classical knotted surface theory.  In \cite{PLC_surface}, Lambert-Cole uses the notion of transverse bridge position to characterize symplectic surfaces in $\mathbb{C}\mathrm{P}^2$. 

\subsection{Main results} 
The overarching goal of this paper is to answer the following:

\begin{question}\label{qn: symp_bridge}
    Can we place a symplectic surface $S$ in a Weinstein domain $W$ in 
    bridge position respecting the geometry of  a multisection-with-divides decomposition of $W$?
\end{question}
A positive answer to this question would imply that symplectic surfaces can be encoded combinatorially by a system of curves on a surface.
In this article, we focus on a particular family of symplectic surfaces inside Weinstein domains, known as {\em positive ascending} surfaces (Def.~\ref{pos_ascending.def}). In Theorem~\ref{thm: main}, we show that these can be placed in bridge position with respect to a multisection-with-divides decomposition. 
In some cases, it is known that symplectic surfaces are always positive ascending (up to isotopy); we elaborate on this in Section~\ref{surface_categories.sec}.

To prove Theorem~\ref{thm: main}, we introduce the notion of a {\it transverse banded unlink} in a Weinstein domain (Sec.~\ref{sec:transverse_banded_unlinks}), and show that positive ascending surfaces are naturally described by such banded unlinks.  We then introduce a notion of bridge position for transverse banded unlinks, and show that a bridge-position transverse banded unlink gives rise to a bisection of the original positive ascending surface.  In the case where the Weinstein domain is $B^4$ with its unique Weinstein structure, this bisection can be described particularly simply, by a {\em braided tri-plane diagram} (Def.~\ref{def: braided_tri-plane}).

The following theorem can be viewed as a partial converse to Theorem~\ref{thm: main}.

\begin{theorem}\label{thm: intro_characterization}
    Let $(W,S)$ be a pair where $W$ is a Weinstein domain equipped with a bisection-with-divides decomposition, and $S$ is a symplectic surface in bridge-bisected position with respect to that decomposition. 
    Then $S$ is symplectically isotopic (not necessarily rel boundary) to a positive ascending surface in $W$, with respect to the Weinstein structure compatible with the bisection-with-divides.
\end{theorem}
We note that the isotopy in the statement of Theorem~\ref{thm: intro_characterization} includes boundary stabilizations (Def.~\ref{boundary_stab.def}) and, in particular, involves transverse isotopy of the boundary transverse link.

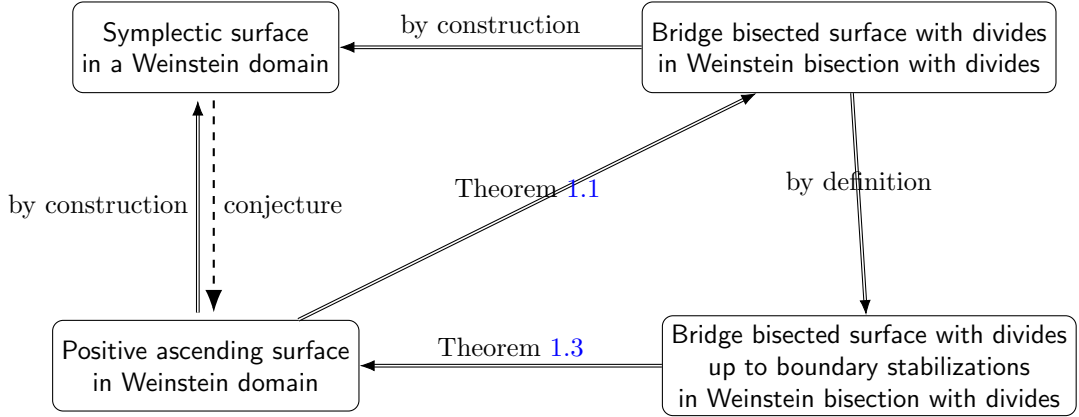
\begin{figure}[h]
\centering
\begin{tikzpicture}[
    box/.style={
        draw,
        rounded corners,
        minimum width=3cm,
        minimum height=1.2cm,
        align=center,
        font=\sffamily
    },
    arrow/.style={
        thick,
        -{Latex[length=3mm]}
    }
]

\node[box] (A) {Symplectic surface\\ in a Weinstein domain};
\node[box, right=4cm of A] (B) {Bridge bisected surface with divides \\in Weinstein bisection with divides};
\node[box, below=3cm of A] (C) {Positive ascending surface\\ in Weinstein domain};
\node[box, right=4cm of C] (D) {Bridge bisected surface with divides\\ up to boundary stabilizations\\ in Weinstein bisection with divides};
\draw[double,-{Latex}] (B) --node[midway, above] {by construction} (A);

\draw[arrow, dashed,shorten >=20pt, shorten <=20pt] ($(A) + (3pt,0)$) --node[midway, right] {conjecture} ($(C) + (3pt,0)$);

\draw[double,-{Latex},shorten >=20pt, shorten <=20pt] ($(C) + (-3pt,0)$) --node[midway, left] {by construction} ($(A) + (-3pt,0)$);
\draw[double,-{Latex}] (B) --node[midway, above] {by definition} (D);
\draw[double,-{Latex}] (C) --node[midway, above] {Theorem~\ref{thm: main}} (B);
\draw[double,-{Latex}] (D) --node[midway, above] {Theorem~\ref{thm: intro_characterization}} (C);
\end{tikzpicture}
\caption{Relationships between symplectic, positive ascending, and bisected surfaces with divides in a Weinstein domain.}
\label{surface_flowchart.fig}
\end{figure}

Theorems~\ref{thm: main} and \ref{thm: intro_characterization} can be considered 4-dimensional generalizations of Pavelescu's result \cite{pavelescu2012braiding} that any transverse knot in a contact 3-manifold can be braided in a supporting open book, and any such braid closure corresponds to a transverse knot. Figure \ref{surface_flowchart.fig} summarizes the relationships between positive ascending and symplectic surfaces in bridge position.

Weinstein domains are equivalent to Stein domains \cite{Cieliebak-Eliashberg-stein-weinstein-book}, and thus are equipped with a complex structure. Complex curves, i.e.\ complex codimension 1 curves inside Stein domains, are all positive ascending \cite{hayden21}. Hence the following corollary holds.

\begin{corollary}\label{cor: intro_complex}
    Any complex curve in a 4-dimensional Stein domain $(W, J)$ can be placed in a bridge-bisection-with-divides position with respect to a bisection-with-divides decomposition compatible with a Weinstein structure on $(W, J)$.
\end{corollary}
    
It is not known if every symplectic surface in a Weinstein domain is isotopic through symplectic surfaces to a complex curve, or if it is symplectically isotopic to a positive ascending surface. 
It is true in the case of the 4-ball by Boileau--Orevkov \cite{boileau-orevkov}, and in the case of Weinstein domains supported by planar nearly-Lefschetz fibrations \cite{BRW:PANLF}, that any symplectic surface is quasipositive, and hence positive ascending. Some of these results are also claimed by work-in-progress of Baykur--Etnyre--Hayden--Hedden--Kawamuro--Van Horn-Morris. 
Thus the following is a consequence of Theorems~\ref{thm: main} and \ref{thm: intro_characterization}; restricted to the 4-ball (which is covered by the result), this can be regarded as a relative version of the main theorem of \cite{PLC_surface}:

\begin{corollary}\label{cor: intro_b4}
    Any symplectic surface inside a Weinstein domain supported by a planar nearly-Lefschetz fibration can be placed in a bridge-bisection-with-divides position with respect to a bisection-with-divides decomposition. Conversely, any smooth surface that can be represented by a braided tri-plane diagram for a multisection-with-divides decomposition of the 4-ball is a symplectic surface.
\end{corollary}

In~\cite{baykur_VHM_infinite}, Baykur and Van Horn-Morris produced examples of infinitely many pairwise distinct symplectic fillings in $B^4$ of the same quasipositive braid closure in the 3-sphere. In Example~\ref{Baykur_VHM_examples}, we describe how to obtain braided tri-plane diagrams for these surfaces; see Figure~\ref{fig:BVM_5}.

The foundational work of \cite{hayden21} states that positive ascending surfaces can be represented by quasipositive factorizations in positive allowable open books. Our Theorem~\ref{thm: main} says that such factorizations admit simple decompositions as compositions of two quasipositive pointed mapping classes, each representing braided unlinks with maximal self-linking number. Along the way, we obtain several geometric and combinatorial descriptions of positive ascending surfaces, which yield decompositions of those surfaces into simple pieces:

\begin{theorem}\label{thm:equivalent_descriptions}
    Let $K$ be a quasipositive transverse knot in the boundary of a Weinstein domain $(W,\omega)$, with  associated pointed mapping class $(\phi, F)$. The following are 
    descriptions of a positive ascending surface $S$ in $W$ with boundary $K$:
    \begin{enumerate}
    \item[(BUD)] a banded unlink diagram for $S$ in a Kirby--Weinstein diagram for $(W,\omega)$; 
        \item[(QPF)] a quasipositive factorization of $\phi$;
         \item[(BD)] a bisection with divides, i.e.\ a decomposition of $S$ into two trivial symplectic disk systems whose boundaries are \textmaxsl unlinks in a connect sum of standard $S^1 \times S^2$'s;
         \item[(SDD)] a shadow diagram with divides, i.e., an arc-and-curve system on a closed orientable surface satisfying the criteria in Definition~\ref{shadowdiagram.def};
        \item[(PMC)] a factorization of $\phi$ into two quasipositive pointed mapping classes, i.e., $(\phi_1, F_1)$ and $(\phi_2, F_2)$ such that they each correspond to \textmaxsl unlinks in a connect sum of standard $S^1 \times S^2$'s as in Definition~\ref{def:PMC}.
    \end{enumerate}
    
\end{theorem}
\begin{proof}
    The correspondence with BUD and QPF was done in \cite{etnyregolla} and \cite{hayden21}, respectively. Section~\ref{sec:proofs} and Propositions~\ref{prop:SDD} and \ref{prop:PMC} discuss the equivalences with BD, SSD, and PMC diagrams, respectively. 
\end{proof}

There are algorithms relating these descriptions. We take as input a banded unlink diagram for $S$ in a Kirby--Weinstein diagram for $(W,\omega)$.  We then place this diagram in banded bridge position to obtain a bisection with divides for $S$. 
This bisection can be encoded in a shadow diagram with divides, from which one can read off the factorization of $\phi$ into two quasipositive pointed mapping classes.
In Example~\ref{2_compo_unlink_one_band_example}, we run through these procedures for a simple symplectic disk inside the cotangent disk bundle of $S^2$. 
Figure~\ref{fig:bandedinkwdiagram_1} shows a BUD for this surface, which is braided in Figure~\ref{fig:braidedbandedinkwdiagram}; SDD and PMC descriptions are shown in Figures~\ref{fig:shadowdiagram} and~\ref{fig:pointedmonodromy}, respectively.

Another well-known construction of symplectic surfaces in Weinstein domains are multisections\footnote{The term ``multisection'' has different meanings in the context of trisections and Lefschetz fibrations; the intended meaning will be clear from context.} of Lefschetz fibrations \cite{baykur2016multisections}.
A multisection $\Sigma$ for a Lefschetz fibration $\Pi: W \to D^2$ is a surface for which $\Pi|_\Sigma$ is a branched covering. 
As noted in the following corollary, they are also positive ascending with respect to the Weinstein structure compatible with the Lefschetz fibration, and hence admit bridge-bisections-with-divides decompositions with respect to the associated bisection with divides.

\begin{corollary}\label{cor: intro_multisection}
    Let $\Sigma$ be a multisection of a Lefschetz fibration $\Pi: W \to D^2$, where $W$ is a Weinstein domain. Then $\Sigma$ can be placed in a bridge-bisection-with-divides position with respect to the bisection-with-divides decomposition associated to this Weinstein structure.
\end{corollary}

\begin{proof}
    The branch points of $\Pi|_\Sigma$ map to the interior of $D^2$, and can all be assumed to be simple. As discussed in \cite[Section 3]{baykur2016multisections}, near such a branch point, locally the map looks like $\Pi:\mathbb{C}^2 \to \mathbb{C}$ given by $\Pi(z_1,z_2)=z_1$, such that $\Sigma$ is locally parametrized as $\{(z^2,z)\}$. 
    The Weinstein structure on $W$ in this neighborhood is encoded by the Morse function given by the real part of $\Pi$. Restricted to $\Sigma$, this clearly restricts to a Morse function and the branch point corresponds to a positive index-1 critical value. 
    Also, we can assume that the Weinstein Morse function on $W$ has a single index-0 critical point, and in a neighborhood of that critical point, we can parametrize $\Sigma$ so that the Morse function restricted to $\Sigma$ has $n$ index-0 critical points, where $n$ is the degree of the multisection. Away from the branch points, it can be seen that the real part of $\Pi$ restricts to a Morse function with no critical points on a multisection $\Sigma$. Further, it follows from \cite{baykur2016multisections} that the intersection of $\Sigma$ with any level set of this Morse function is a quasipositive link. It thus follows that $\Sigma$ is a positive ascending surface, and so by Theorem~\ref{thm: main}, it can be placed in a bridge position with respect to the bisection-with-divides decomposition of $W$.
\end{proof}

Any Weinstein domain is a branched cover over the standard Weinstein 4-ball, branched over a symplectic surface \cite{LP_branched}. In Section~\ref{sec: palfs_branched_covers}, we explain how a multisection-with-divides decomposition can be pulled back to the branched cover if the branch locus is in bridge position. Thus, using Corollary~\ref{cor: intro_b4}, we give an alternate proof of the main theorem of \cite{IS-divides}, that any Weinstein 4-manifold admits a multisection-with-divides decomposition (Cor.~\ref{Cor: existence_multisections_Weinstein}).

\subsection{Acknowledgements}
This project began at the 2023 Trisectors' Workshop held at UC Davis.
We thank Marion Campisi, James Hughes, and Daniela Cortes Rodriguez for initial collaboration, and the organizers for an enriching and productive workshop. 
We also thank Inanc Baykur, John Etnyre, Kyle Hayden, Gabriel Islambouli, and Laura Starkston for helpful correspondence. 
Part of this work was completed at PCMI 2026. PC was partially supported by NSF DMS-2145384. RA and AR were both partially supported by an AMS-Simons Travel Grant.

\section{Background on contact and symplectic manifolds} \label{sec: background}

Throughout this work, all manifolds will be smooth and orientable, unless otherwise stated.  We use $(W,\omega)$ to denote a Weinstein domain with symplectic structure $\omega$ (see Sec.~\ref{weinstein.sec}), $(M,\xi)$ to denote a closed 3-manifold with contact structure $\xi$.  We assume familiarity with contact structures, and Legendrian and transverse knots, and refer the reader to \cite{etnyre2003introductory} for an overview of these topics.

\subsection{Weinstein Domains}\label{weinstein.sec}

A \emph{Weinstein domain} is a real 4-dimensional Liouville domain $(W,\lambda)$, together with a Morse function $\rho:W\rightarrow \mathbb{R}$ such that the Liouville vector field for $\lambda$ is a gradient-like vector field with respect to $\rho$.  As a result, Weinstein domains have a handle structure compatible with the symplectic structure $\omega$, where $\omega=d\lambda$. We will often abuse notation and refer to a Weinstein domain using the symplectic form as $(W,\omega)$ instead of $(W,\lambda)$, and also suppress the Morse function $\rho$ from the notation.

The simplest example of a Weinstein domain, in four dimensions, is the standard 4-ball $B^4$ with $\lambda_{\text{st}} = x_1dy_1 + x_2dy_2$, and $\rho_{\text{st}} = x_1^2+y_1^2+x_2^2+y_2^2$. The next simplest example is a 4-dimensional 1-handlebody $\natural^k S^1\times B^3$ with $\rho$ having one and $k$ critical points of index $0$ and $1$,  respectively. 
Such a domain is the unique minimal symplectic filling of the standard contact structure on $\#^k S^1\times S^2$ \cite{mcduff_rational-ruled}. It is known that 1-handlebodies 
are also unique as minimal weak fillings, and deformation equivalent to minimal strong fillings and Stein fillings~\cite{Wendl_strongly}.

By work of Gompf \cite{gompf1998handlebody},
the handlebody structure on a Weinstein domain can be encoded by a \emph{Kirby--Weinstein diagram} $(\Lambda,k)$, which is a Legendrian link $\Lambda$ in the 3-manifold $\#^k S^1\times S^2$ endowed with its standard contact structure $\xi_{std}$. The corresponding Weinstein domain is obtained by attaching 2-handles to $\natural^k S^1\times D^3$ along the components $\Lambda_i$ of $\Lambda$ with framing equal to $tb(\Lambda_i)-1$. We will use the front projection on $\#^k S^1\times S^2$ to draw Weinstein diagrams; see Figure~\ref{fig:bandedinkwdiagram_1} for an example.

\begin{definition}\label{def:unlink}
    A transverse link $L \subset (M, \xi)$ is called a \emph{transverse unlink with maximal self-linking}, or a \emph{\textmaxsl unlink}, if there is a transverse isotopy taking it to a standard unlink\footnote{A standard unlink in $(\mathbb{R}^3,\xi_{std})$ is one whose components each have self-linking number -1.} in a Darboux neighborhood inside $M$.
\end{definition}

The following lemma shows that the notion of a transverse unlink with maximal self-linking is well-defined.

\begin{lemma}\label{lem:unique_unlink}
    If $(M,\xi)$ is connected, any two transverse unlinks with maximal self-linking are transverse isotopic iff they have the same number of components.
\end{lemma}

\begin{proof}
    Let $L$, $L'$ be two such unlinks with the same number of components. There exists Darboux balls $B$, $B'$ respectively such that $L$ and $L'$ are transverse isotopic to standard unlinks in them. As $(M, \xi)$ is connected, after an ambient isotopy, we can assume that $B$ and $B'$ are the same. Then, by the transverse simplicity of unlinks in the standard $S^3$ \cite{eliashberg_fraser}, we are done.
\end{proof}

A \textmaxsl transverse unlink bounds a collection of symplectic disks in any symplectization neighborhood of $(M, \xi)$. 
In particular, if $(M, \xi)$ is fillable, it bounds disks in any filling. Proposition~\ref{uniquedisks.prop} shows that such disk fillings are unique inside 1-handlebodies provided that the unlinks have maximal self-linking numbers.

\begin{proposition}\label{uniquedisks.prop}
Let $W=\natural^k S^1\times D^3$ be a Weinstein 1-handlebody, and let $L$ be a transverse \textmaxsl unlink in $\partial W$. Up to symplectic isotopy fixing $L$, the unlink $L$ bounds a collection of symplectic disks in $W$ unique up to symplectic isotopy. 
\end{proposition}

\begin{proof}

From \cite[Thm.~1.1.$\text{A}'$]{Eliashberg95},
any transverse unknot in $\partial B^4$ with the standard contact structure is uniquely filled by a symplectic disk (unique up to symplectic isotopy rel boundary). 
Now consider $L$ as above. Since each component $U_i$ of $L$ has maximal self-linking number, there are symplectic disks $D_i\subset V$ with $\partial D_i=L$.  The disks $D_i$ are unique up to smooth isotopy by \cite{kamada2002braid,meier2018bridge}, and up to symplectic isotopy by the argument in \cite{Eliashberg95}, which can be applied to transverse unlinks in the boundary of any minimal strong filling. 
As Weinstein 1-handlebodies are minimal strong fillings, the result follows.  
\end{proof}

\begin{remark}\label{rem:disk_isotopy}
    By \cite[Thm.~2.9]{etnyre_knots}, the transverse isotopy of Lemma~\ref{lem:unique_unlink}  extends to a symplectic isotopy of the disks bounded by the link. 
\end{remark}

\subsection{Open Books and Contact Heegaard Splittings}\label{section:contact heegaard}

Let $M$ be a fixed orientable closed 3-manifold. Contact structures in $M$ can be codified using topological information such as open book decompositions and contact Heegaard splittings.
We can translate between these two descriptions, via algorithms that we describe below.

An \emph{open book decomposition} of $M$ is a pair $(\Gamma,\pi)$ where $\Gamma$ is an oriented link in $M$ and $\pi:M\setminus \Gamma \rightarrow S^1$ is a fibration of the complement of $\Gamma$ such that $\pi^{-1}(t)$ is the interior of a compact surface $F_t$ with boundary equal to $\Gamma$. 
The surfaces $F_t$ are called \emph{pages} and $\Gamma$ is the \emph{binding} of the open book. We say an open book supports the contact structure $\xi$ if up to an isotopy through contact structures, $\xi = \ker \alpha$ for some one-form, $d\alpha$ is positive on the pages $F_t$, and $\alpha$ is positive on the binding $\Gamma$. Every open book supports a contact structure $\xi_\pi$ unique up to isotopy on $M$ \cite{thurston_winkelnskemper_obd} that restricts to a universally tight contact structure on $M \setminus \Gamma$ \cite{etnyre_vv_torsion}.

Conversely, by the Giroux correspondence any contact structure on $M$ has a supporting open book decomposition, and any two open books for the same contact structure are related by moves called stabilization or destabilization.

A surface $\Sigma\subset (M,\xi)$ is \emph{convex} if there is a contact vector field $v$ for $\xi$ that is defined in a neighborhood of $\Sigma$ and is transverse to $\Sigma$. 
The \emph{dividing set} for a convex surface $\Sigma$ is the collection of closed curves 
$\Gamma = \{x\in\Sigma|v_x \subset \xi\}$. 
An example of a convex surface is the union of two pages of an open book $F_0\cup \overline{F}_{1/2}$;
in this case, the dividing set is equal to the binding $\Gamma=\partial F_t$ (see e.g.\ \cite{etnyre_vv_torsion}).  

A Legendrian graph $L\subset (M,\xi)$ has a neighborhood unique up to contactomorphism by Theorem 2.5.8 of \cite{geiges2008introduction}. A \emph{contact handlebody} $(H,\xi)$ is a standard contact neighborhood of a Legendrian graph. A \emph{contact Heegaard splitting} of $(M,\xi)$ is a decomposition $M=H_1\cup_\Sigma\overline{ H}_2$ such that 
\begin{itemize}
    \item $\Sigma := H_1 \cap \overline{H}_2$ 
    is a convex surface with dividing set $\Gamma$, and 
    \item each $H_i$ is a contact handlebody in $(M,\xi)$. 
\end{itemize} 
Contact Heegaard splittings were introduced by Giroux in~\cite{giroux2003g} and further developed by Torisu in~\cite{torisu2000convex}. It is a fact that $H_1$ and $H_2$ are two ``halves" of an open book decomposition $(\Gamma, \pi)$ supporting $(M, \xi)$. The surface $\Sigma$ is equal to the union of two pages $F_{0}\cup \overline{F}_{1/2}$ glued along their boundaries $\Gamma$. 

\begin{definition}
Let $(H,\xi)$ be a contact handlebody, and let $\Sigma = \partial H$ be its convex boundary. 
A \emph{half-open book}  supporting $\xi$ is a pair $(\Gamma, \pi)$ where 
    \begin{itemize}
        \item $\Gamma$ is the dividing set on $\Sigma$, and
        \item $\pi: H \backslash \Gamma \to [0,\frac{1}{2}]$ is a fibration where every $\pi\inv(t)$ is the interior of a compact surface $F_t$ with $\partial F_t = \Gamma$.  These are the \emph{pages} of the half-open book.
    \end{itemize}
In particular,  $\Sigma = F_0 \cup_\Gamma \overline{F}_{1/2}$.
\end{definition}

When referring to half of a standard open book decomposition $(\Gamma,\pi)$ of $(M,\xi)$ with $M=H_1\cup \overline{H}_2$ a contact Heegaard splitting, we use $\pi$ to denote both the projection map for the full open book and the corresponding half-open books for $(H_i,\xi)$. We then have two fibrations, $\pi:H_1\setminus\Gamma\rightarrow[0,1/2]$, and $\pi:H_2\setminus\Gamma\rightarrow[1/2,1]$, where $[1/2,1]$ replaces $[0,1/2]$ in the definition of half-open book in the natural way.

A collection of arcs in a compact surface that cut it into disks is called an \emph{arc system}. A collection of simple closed curves in a closed surface that cut it into a connected planar surface is called a \emph{cut system}. Given a half-open book for a contact handlebody $(H,\xi)$, an arc system for the page $F_0$ induces a cut system for $\partial H$. Such cut systems have the property that each curve intersects the dividing set in two points, and so they remember the contact structure in $(H,\xi)$. This is exploited below to encode a contact Heegaard splitting diagrammatically using curves in the Heegaard surface $\Sigma$. 

\begin{definition}
    A \emph{contact Heegaard diagram} (or Heegaard diagram with divides) is a tuple $(\Sigma, \alpha, \beta, \Gamma)$ such that $\alpha=\{\alpha_i\}$ and $\beta=\{\beta_i\}$ are cut systems for $\Sigma$, and $\Gamma$ is a multicurve cutting $\Sigma$ into two homeomorphic surfaces with boundary and such that the intersection of the $\alpha_i$ ( resp.\ $\beta_i)$ with each half of $\Sigma\setminus \Gamma$ is an arc system.
\end{definition} 

Consider a Weinstein cobordism $(W, \lambda)$ between $(M_-,\xi_-)$ and $(M_+, \xi_+)$, with Legendrian links $\Lambda_\pm \subset (M_\pm,\xi_\pm)$. 
\begin{definition}\label{compatible_W_cobord.def}
    We call open books $(\Gamma_\pm, \pi_\pm)$ supporting $(M_\pm, \xi_\pm)$ compatible with the Weinstein cobordism $((M_-,\xi_-), \Lambda_-)\to((M_+, \xi_+), \Lambda_+)$ if the following hold:
    \begin{enumerate}
        \item The open books $(\Gamma_\pm, \pi_\pm)$ support the contact structures $(M_\pm,\xi_\pm)$.
        \item The cobordism $(W,\omega)$ is obtained by attaching Weinstein 2-handles along $\Lambda_-$.  We call these curves $\Lambda_-$ the {\it vanishing cycles.}
        \item The link $\Lambda_-$ is included in the core of one of the contact handlebodies corresponding to the contact Heegaard splitting corresponding to $(\Gamma_-, \pi_-)$.
        \item The resulting contact Heegaard splitting of $(M_+,\xi_+)$, obtained by Legendrian surgery on the above Heegaard splitting of $(M_-,\xi_-)$, can be stabilized to ensure $\Lambda_+$ is included in the core of one of the handlebodies.
    \end{enumerate}
\end{definition}

It follows from an algorithm of Avdek \cite[Thm.~1.9]{avdek_surgery_obd}, and also the proof of the main theorem by Islambouli--Starkston \cite[Thm.~3.1]{IS-divides}, that given any Weinstein cobordism $W: ((M_-,\xi_-), \Lambda_-)\to((M_+, \xi_+), \Lambda_+)$, and Legendrian links $\Lambda_\pm$ as above,  there exist compatible open books $(\Gamma_\pm, \pi_\pm)$ supporting $(M_\pm, \xi_\pm)$.  

\subsection{Braids, pointed open books, and shadows} \label{subsec: braids} 
We next review relevant diagrammatic tools for representing transverse links in contact manifolds. 

In the sequel, we use the term ``braid'' to refer to both braids and their closures, as follows.
A (closed) {\it braid} in an open book decomposition  $(\Gamma, \pi)$ of $M$ is a link in $M$ that is everywhere transverse to the pages $F_t = \pi\inv(t)$.
Analogously, a \emph{braid} in a half-open book $(\Gamma, \pi)$ for a handlebody $H$ is a properly embedded collection of $n \in \N$ arcs $b_i: [0,1/2] \to H$ for $i \in \{1, \ldots, n\}$, everywhere transverse to the pages $F_t$, with each $b_i(0) \in int(F_0)$ and $b_i(1/2) \in int(F_{1/2})$.

\begin{theorem}[\cite{pavelescu2012braiding}]\label{thm:pavelescu_braids}
Any transverse link in a contact $3$-manifold $(M, \xi)$ with compatible open book $(\Gamma, \pi)$ can be transversely isotoped to a braid in $(\Gamma,\pi)$.
\end{theorem}

Therefore when $(M, \xi)$ is equipped with a contact Heegaard splitting $M=H_1\cup_{\Sigma_g}H_2$, any transverse link $L$ in $(M, \xi)$ can be transversely isotoped into \emph{transverse bridge position}, meaning that each $\tau_i=L\cap H_i$ is a trivial tangle (and more specifically, a braid) in the half-open book structure supporting $(H_i,\xi)$ for $i=1,2$. 
Although the concepts of `closed braids' and `links in transverse bridge position' are synonyms for a fixed contact 3-manifold, we use them separately to emphasize the underlying decomposition of $(M,\xi)$: open books for braids, and contact Heegaard splittings for transverse bridge splittings. These perspectives provide us with two combinatorial descriptions of transverse links called \emph{pointed open books} and \emph{shadow diagrams}, which we describe below.

Each open book decomposition of $M$ corresponds to an abstract open book $(F,\phi)$, where $F$ is an oriented surface with boundary and $\phi:F\rightarrow F$, the monodromy of the open book, is a self-diffeomorphism of $F$ which is the identity near $\partial F$.  A braid in an abstract open book $(F,\phi)$ can be described by a \emph{pointed open book}, which is a triple $(F,P,\hat\phi)$ such that $P$ is a finite set of points in $F$, and $\hat\phi$ is a diffeomorphism of pairs $\hat\phi:(F,P)\rightarrow (F,P)$, fixing  $P$ setwise.  In addition, $\hat\phi$ is called the \emph{pointed monodromy of the pointed open book}.  To ensure the pointed open book describes a braid in a given abstract open book $(F,\phi)$, we require that $\hat\phi$ is isotopic to $\phi$ via an isotopy fixing $F$ near $\partial F$.

\begin{notation}
The monodromy $\phi$ of a pointed open book can be decomposed into elementary pieces of the following two types (see Figure~\ref{fig:twist} for conventions used in \cite{hayden21} and in this work): 
\begin{itemize}
    \item \emph{Dehn twist} $D_\gamma$ along a closed curve $\gamma$ that avoids the marked points
    \item \emph{half twist} $H_\alpha$ along an arc $\alpha$ such that $\partial \alpha$ consists of two distinct marked points.
\end{itemize}
Furthermore, a \emph{point-push} (or a \emph{push map} in \cite{hayden21}) along a closed arc $\delta$ with both endpoints on a marked point $p$ (and otherwise avoiding marked points) is a compound monodromy $D_{\gamma^+}\circ D^{-1}_{\gamma^-}$ where $\gamma$ is as in Figure~\ref{fig:twist}.
\end{notation}

\begin{figure}[h]
    \centering
    \includegraphics[width=0.8\linewidth]{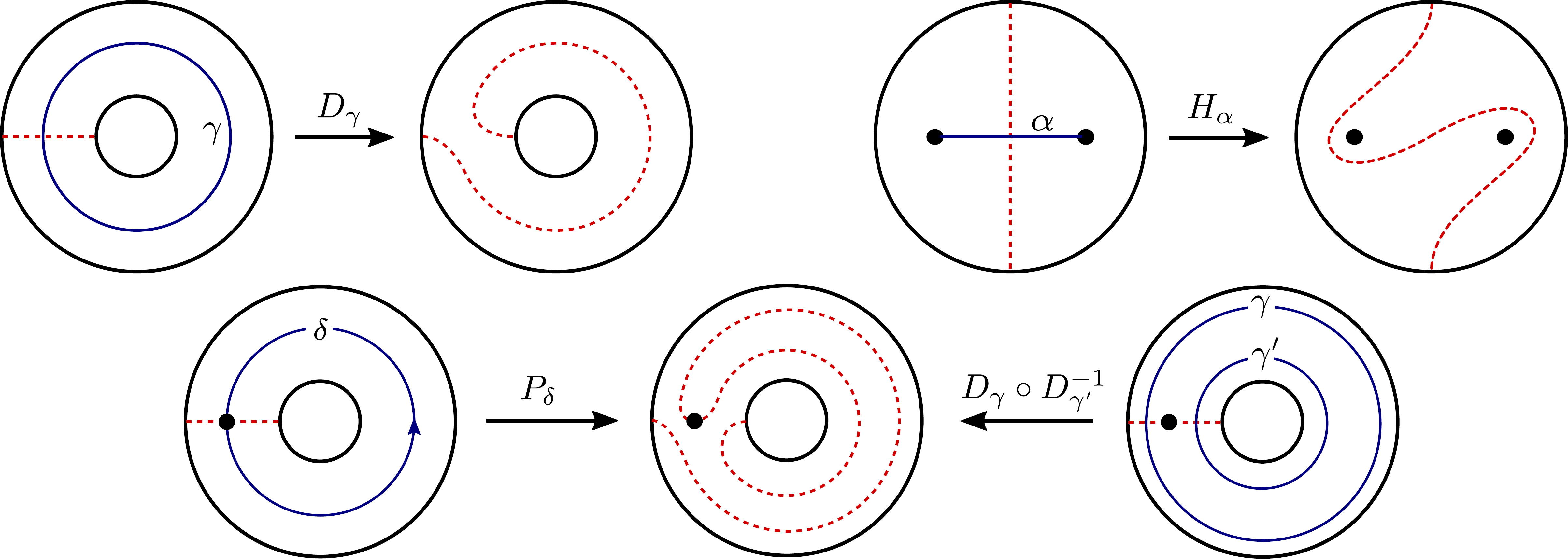}
    \caption{(left) Dehn twist, (right) half-twist, and (bottom) point-push maps. We thank Kyle Hayden for allowing us to import this figure from \cite{hayden21}.}
    \label{fig:twist}
\end{figure}

A \emph{quasipositive braid} is a braid whose pointed monodromy can be written as a composition of positive half-twists and arbitrary Dehn twists.
We are particularly interested in the subclass of \emph{PA-quasipositive braids}\footnote{`PA' is short for `positive allowable'.}, which we introduce next.  
These bound properly embedded symplectic surfaces in a compatible Stein filling of the contact 3-manifold \cite[Thm.~1.2]{hayden21}.
A key observation is that a single point-push along a vanishing cycle in a compatible open book (Def.~\ref{compatible_W_cobord.def}) is null-homologous and bounds a symplectic disk in the cobordism.  The self-linking number can be used to obstruct symplectic fillability for higher-order point-pushes.  For example, two point-pushes in the annulus open book for tight $S^3$ is an unknot with self-linking $-2$ (see e.g.\ \cite[Sec.~2.6.4]{etnyre_knots}), and hence does not bound a symplectic disk in the 4-ball.

\begin{definition}\label{def:qp} Let $(M, \xi)$ be a contact 3-manifold which is the boundary of a Weinstein domain $(W,\omega)$, and a Legendrian link $\Lambda \subset (M, \xi)$. Consider an open book $(\Gamma, \pi)$ compatible with the Weinstein cobordism from $((\#^k S^1 \times S^2,\xi_{st}),\Lambda_0)$ to $((M, \xi),\Lambda)$.  A braid in the open book $(\Gamma,\pi)$ is \emph{PA-quasipositive} if it can be encoded by an abstract pointed open book whose pointed monodromy is isotopic to a product of positive half-twists about arbitrary arcs, positive Dehn twists about arbitrary curves, and at most one point-push about each vanishing cycle (as in Definition \ref{compatible_W_cobord.def}) that lies on a page of $(\Gamma,\pi)$. 
Such a pointed monodromy is called a {\it PA-quasipositive pointed mapping class.}
\end{definition}

Our notion of PA-quasipositive is different from the notion of a quasipositive braid in an arbitrary open book in \cite[Def.~2.8]{hayden21}, which allows arbitrary Dehn twists and does not explicitly mention point-pushing.  However, a PA-quasipositive braid is a quasipositive braid in a positive allowable open book as defined in the paragraph after \cite[Def 2.12]{hayden21}, as each of the above point-pushing maps can be re-written as a product of a positive and a negative Dehn twist; because each point-push occurs along a vanishing cycle, the negative Dehn twist can be canceled out by a positive Dehn twist in the monodromy factorization of $(\Gamma,\pi)$. Each such point-push corresponds to the birth of a symplectic disk in  $(\#^k S^1\times S^2)\times I$. Note we do not allow multiple point-pushes about a single vanishing cycle; in general a braid corresponding to a multiple point-push does not have maximal self-linking number  and hence cannot bound a symplectic surface. 
 
A trivial tangle in a handlebody $H$ can be represented by a collection of disjoint paths called \emph{shadows} in $\partial H$, namely the images of the tangle arcs under an isotopy into $\partial H$.  In particular, a braid in a half-open book with pages $F_t$ is a trivial tangle, and can be represented by a set of shadows on $\Sigma=F_0\cup_\Gamma \overline{F}_{1/2}$ such that each shadow intersects the dividing set $\Gamma$ once.  
A transverse link in $(M,\xi)$ transverse bridge position with respect to an open book supporting $\xi$ is then represented by a pair of shadows on $\Sigma=F_0\cup_\Gamma \overline{F}_{1/2}$.  Such a pair corresponds to a quasipositive (resp.\ PA-quasipositive) link whenever the first set of shadows maps to the second by a quasipositive (resp.\ PA-quasipositive) mapping class.

\section{Symplectic surfaces in Weinstein domains}

We are now ready to incorporate symplectic surfaces into a Weinstein domain. We work with the subclass of properly embedded symplectic surfaces which are ascending and have only positive critical points (Def.~\ref{ascending.def}). Proposition~\ref{prop: construct_pos_asc} in this section shows that positive ascending surfaces are diagrammatically described by transverse banded unlinks in a Kirby--Weinstein diagram $(\Lambda,k)$ (see Sec.~\ref{weinstein.sec}) for $(W,\omega)$ with boundary $(\partial W, \xi)$. 
Transverse banded unlink diagrams are introduced in Section~\ref{sec:transverse_banded_unlinks} and are the input for all of the examples in this work.

\begin{definition}\label{ascending.def}
Let $(W, \lambda)$ be a Weinstein domain with a compatible Morse function $\rho:W\rightarrow \R$.
A smoothly embedded oriented surface $(S, \partial S) \subset (W, \partial W)$ is \emph{ascending} if 
\begin{itemize}
    \item $S$ contains no critical points of $\rho$,
    \item $\rho|_S$ is a Morse function, and
    \item for each regular value $c$, $\rho|_S^{-1} (c)$ is a positively transverse link in $\rho^{-1}(c)$. 
\end{itemize}

\end{definition}

To define a positive ascending surface, we recall some background from \cite[Sec 4.1]{hayden21}. Let $J$ be an almost complex structure compatible with $(W,\omega)$, and let $S$ be a properly embedded ascending surface in $W$. 
The critical points of $S$ are complex with respect to $J$, i.e., given a critical point $p$, $T_pS$ is a complex subspace of $T_pW$. Further, a critical point is said to be \emph{positive} or \emph{negative} depending on whether the intrinsic orientation on $T_pS$ agrees or disagrees, respectively, with the complex orientation. 
 
\begin{definition}[\cite{hayden21}]\label{pos_ascending.def}
A \emph{positive ascending surface} is an ascending surface with only positive critical points. 
\end{definition}

The work of Hayden in \cite{hayden21} 
shows that a positive ascending surface is isotopic, through compactly supported isotopy of ascending surfaces, to a symplectic surface. Theorem 4.2 of \cite{hayden21} shows that positive ascending surfaces in Stein domains have boundary transversely isotopic to a quasipositive braid. 
Thus, morally, positive ascending  surfaces are those described by movies of quasipositive braids. We further discuss this idea in Section~\ref{surface_categories.sec}.

\subsection{Transverse banded unlinks}\label{sec:transverse_banded_unlinks}

Let $L$ be a transverse link in $(M, \xi)$. A \emph{(3-dimensional, transverse) band} $v$ is an embedding of a square $I\times I$ into $M$, meeting the link $L \subset \{0\} \times M$ at the two arcs $v(\{0,1\}\times I)$, such that, on the interiors of its arcs, $\partial v$ is embedded and transverse to $\xi$, and such that the derivatives of the arcs $v( I\times \{0,1\})$ agree with those of $L$ at their endpoints. The resolved link $L[v]$ is defined to be the surgery of $L$ along $v$, i.e.\ $$ L[v] = L \setminus v(\{0,1\}\times I) \cup v(I \times \{0,1\}).$$
With $L$ as above, and $\nu$ a collection of 3-dimensional transverse bands, we define $L[\nu]$ as the surgery of $L$ along all $v\in \nu$. Note that this definition implies both $L$ and $L[\nu]$ are transverse links, but there are no additional geometric constraints on the bands $\nu$. The middle panel of Figure~\ref{fig:fig_qp_trefoil} depicts examples of 3-dimensional transverse bands in $\R^3$ with the rotationally symmetric contact structure.

A \emph{(4-dimensional, symplectic) band} $v$ is a symplectic embedding of a square $I\times I$ into the piece of the symplectization $[0,1] \times M$, meeting the link $L \subset \{0\} \times M$ at the two arcs $v(\{0,1\}\times I)$, such that the projection to $M$ of $\partial v$ is embedded and transverse to $\xi$ as in the definition of a 3-dimensional transverse band. The resolved link $L[v]$ is defined to be the projection to $M$ of the surgery of $L$ along $v$.  Similarly, for a collection $\nu$ of symplectic 4-dimensional bands $v$, $L[\nu]$ is the projection to $M$ of the surgery of $L$ along each $v\in \nu$.  
Note that not all 3-dimensional transverse bands arise as projections of 4-dimensional symplectic bands, but \emph{positive} 3-dimensional transverse bands do.  
 
\begin{definition}\label{def:band}
We say that a (three-dimensional, transverse) band $v$ is a {\em positive band} if in a front projection of a Darboux neighborhood, $L$ and $L[v]$ are related by the addition of a positive crossing as in Figure \ref{fig:positive}. 
\end{definition}

\begin{figure}[h]
    \centering
    \includegraphics[width=0.5\linewidth]{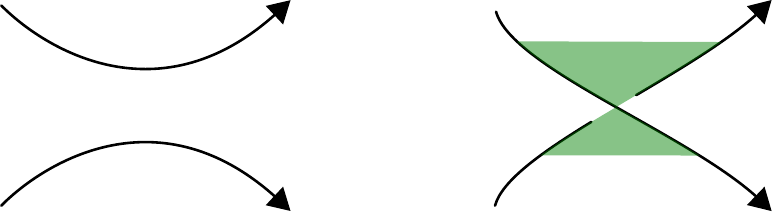}
    \caption{Local model for the addition of a \emph{positive transverse band}. The band surgery occurs in the front projection inside a Darboux ball.}
    \label{fig:positive}
\end{figure}

\begin{remark} \label{bandsurface.rem}
From the proof of \cite[Lem.~2.8]{etnyregolla} it follows that when $v$ is a positive band for $L$ in $(M,\xi)$, the transverse links $L \subset \{0\} \times M$ and $L[v] \subset \{1\} \times M$ cobound a symplectic surface $S_v \subset [0,1] \times M$ satisfying the following properties:
\begin{enumerate}
    \item the projection $\pi_{\R}:[0,1] \times M \to [0,1]$ restricted to $S_v$ is a Morse function,
    \item the function $\pi_{\R}|_{S_v}$ has a single critical value of index one,
    \item the surface $S_v$ is nonsingular with respect to $\pi_{\R}|_{S_v}$ as in \cite[Def.~3.3]{hayden21},
    \item the only singularity of the characteristic  foliation on the projection of $S_v$ to $M$ is a positive hyperbolic point, and
    \item the characteristic foliation on the projection of $S_v$ to $M$ is weakly gradient-like for some Morse function on $S$ as in \cite[Sec.\ 1]{hayden21}.
\end{enumerate}
\end{remark}

The properties in Remark~\ref{bandsurface.rem} are technical and will be needed to use machinery from \cite{hayden21} later in this work. It follows from \ref{item:bb2} that for transverse links with positive transverse bands, the surface $S_\nu$ is a collection of annuli and $|\nu|$ pairs of pants.

A \emph{transverse banded link} in $(M,\xi)$ is a banded link $(L,\nu)$ such that $L$ is transverse to $\xi$ and each $v\in \nu$ is a positive band. A transverse isotopy of a transverse banded link $(L,\nu)$ in $(M,\xi)$ is an isotopy of the banded link $(L,\nu)$ through transverse banded links.
Transverse isotopies of a transverse banded link $(L,\nu)$ extend to symplectic isotopies of the surface $S_\nu$. One can prove this using Theorem 4.3, Theorem 1.3, and Remark 3.6 in \cite{hayden21}.

\begin{definition}
A \emph{transverse banded unlink in a Kirby--Weinstein diagram} is a pair $(U,\nu)\subset (\Lambda,k)$ where $U \subset (\#^k S^1 \times S^2, \xi_{st})$ is a \textmaxsl transverse unlink (Def.~\ref{def:unlink}) and $\nu$ is a collection of positive bands (Def.~\ref{def:band}) disjoint from $\Lambda$.

\end{definition}

An example of a transverse banded unlink in the cotangent bundle of $S^2$ is given in Figure~\ref{fig:bandedinkwdiagram_1}.  Note that the link, together with the attaching sphere for the 2-handle, are both drawn in the front projection, with the link shown as its transverse front projection, and the attaching sphere as its Legendrian front projection.

The following proposition gives an algorithm to construct a symplectic surface from a transverse banded unlink $(U,\nu)$ in a Kirby--Weinstein diagram.  We call this surface the {\it realizing symplectic surface} for $(U,\nu)$ and denote it by $S(U,\nu)$.

\begin{proposition}\label{realizing-symplectic-surface.prop}
There is a realizing symplectic surface $S(U,\nu)$ properly embedded in $(W,\omega)$, constructed from the transverse banded link $(U,\nu)\subset (\Lambda, k)$ via the steps below.
\begin{enumerate}
    \item Consider the pair $(B^4, D)$ of a collection of symplectic 2-disks $D$ bounded by the \textmaxsl unlink with $|U|$ connected components in $\partial B^4$; such disks are provided by Proposition~\ref{uniquedisks.prop}.
    \item Attach $k$ Weinstein 1-handles to $B^4$ away from the boundary unlink of to get the pair \mbox{$(\natural^k S^1 \times B^3,   D)$.}  Identify $\partial D$ with $U$ via a transverse isotopy.
    \item Concatenate the surfaces $S_\nu$ in $(\#^k S^1\times S^2)\times [0,1]$ corresponding to the bands $\nu$ to obtain the pair \mbox{$(\natural^k S^1 \times B^3, s(U,\nu))$.}      
    \item Attach Weinstein 2-handles to $\#^k S^1\times S^2$ along the components $\Lambda_i$ of $\Lambda$ away from the boundary $U[\nu]$ of $s(U,\nu)$. Extend the surface into the cobordism by attaching $(\{t\} \times U[\nu])$ to $s(U,\nu)$. This gives the pair $( W, S(U,\nu))$.  
\end{enumerate}

\end{proposition}

\begin{proof}  
Each cobordism described at each step is symplectic and embedded by construction, and due to the hypothesis that the surfaces are disjoint from the attaching regions of the handles of $W$.
Since being symplectic is an open condition, it follows that $S(U,\nu)$ is a symplectic surface as it is built via the smoothed concatenation of smaller symplectic pieces. The surface thus constructed is well-defined (up to ambient Weinstein homotopy) by Lemma~\ref{uniquedisks.prop}.
\end{proof}

\subsection{Positive ascending surfaces admit transverse banded unlink diagrams}\label{sec:ascending_implies_SUv}

Let $(W,\omega)$ be a Weinstein 4-manifold, and assume that its Morse function $\rho:W \rightarrow \R$ has the property that all its index-2 critical points occur after all its index-1 critical points.
The following proposition states that ascending surfaces can be modified, via an isotopy through positive ascending symplectic surfaces, so that the 
index-1 critical points of $S$ occur between the index-1 and index-2 critical points of $W$; see also Remark \ref{rem: ascending_isotopy}. We are grateful to Laura Starkston for helpful suggestions regarding the following proposition.

\begin{proposition}\label{prop: weinstein_homotopy}
Let $(W,\omega)$ be a Weinstein 4-manifold with an embedded positive ascending surface $(S,\partial S)\subset (W,\partial W)$. 
There is a Weinstein homotopy equivalence of pairs $(W,\omega, \rho;S)$ to $(W',\omega', \rho';S')$ through positive ascending surfaces, where $(S',\partial S')\subset (W',\partial W')$ is a positive ascending surface satisfying 
\[ 
\rho\left(\Crit_1(\rho')\right)
<
\rho\left(\Crit_0(\rho'|_{S'})\right)<\rho\left(\Crit_1(\rho'|_{S'})\right)<\rho\left(\Crit_2(\rho')\right).
\]
\end{proposition}

\begin{proof}

Let $C_i$ and $D_i$ be the index-$i$ critical points of $\rho$ and $\rho|_S$, respectively. 
In the smooth case, this result would be proven in three steps. First, we would take a small perturbation of $S$ to ensure that flow lines passing through $D_0\cup D_1$ are pairwise disjoint and avoid the points in $C_1\cup C_2$. Then, by integrating the flow of $\rho$ we would get a smooth isotopy of $S$ to an embedding $S'$ with the desired properties: the points in $D_0\cup D_1$ above $C_2$ flow downwards and the ones below $C_1$ upwards. The third step is to flow the points in $D_0$ to lie between $D_1$ and $C_1$. This last ambient isotopy is supported on a neighborhood of the flow lines passing through $D_0$.  

For our setup, we need to check that these steps can be done respecting the Weinstein structure of $W$. Suppose the Liouville vector field compatible with the $(W, \omega, \rho)$ is $V$. The first step can be achieved by a small transverse isotopy of the links $S\cap \rho^{-1}(t)$ in regular levels $t \in \rho(D_0)\cup \rho(D_1)$. By the definition of ascending surface, we can ensure that $\rho(D_i) \neq \rho(C_j)$ for any $i$ and $j$. This ensures that $V$ is non-vanishing on $\rho^{-1}(\rho(D_i))$. Then, by Lemma 12.5 of \cite{Cieliebak-Eliashberg-stein-weinstein-book}, there is an isotopy of the Liouville flow $V$ that achieves the trace of our contact isotopy between two regular level sets. Call the new Weinstein structure $(W',\omega',\rho')$, and the new Liouville vector field $V'$. This is a small isotopy, hence it preserves the property of the surface being ascending.

Now, for the second step, we can integrate the backwards flow of $V'$ to flow the points in $D_i$ to $D_i'$, to ensure that $\rho(C_1) < \rho (D_i') < \rho(C_2)$, for $i = 0,1$. 
This induces a scaling of the ambient symplectic form and induces a Weinstein homotopy equivalence of pairs as required.

For the third step, observe that it suffices to consider positive ascending surfaces in $\natural^k S^1 \times B^3$, which are the push-in of immersed Bennequin surfaces in $\#^k S^1 \times S^2$, as in Example 4.8 of \cite{hayden21}. 
Similarly as in that example, we can modify the Morse function $f$ on all the disks to ensure that the bands occur at a higher level {than the births} in the symplectization --- this is intuitively the same as carving out ``deeper disks'' corresponding to each strand of the boundary braid, and then attaching the bands corresponding to half-twists in the braid.
\end{proof}

\begin{remark}\label{rem: ascending_isotopy}
    In the above, we can also ensure to isotope the positive ascending surface so that
    \[\rho\left(\Crit_0(\rho'|_{S'})\right)<\rho\left(\Crit_1(\rho')\right)<\rho\left(\Crit_1(\rho'|_{S'})\right)<\rho\left(\Crit_2(\rho')\right).
    \]
    This can be done by carefully choosing the attaching region of the 1-handles, similar to the proof of \cite[Thm.~4.14]{hayden21}.
\end{remark}

\begin{proposition}\label{prop: construct_pos_asc}
    Up to Weinstein homotopy, a pair $(W, S)$ of a positive ascending surface $S$ in a Weinstein domain $W$ is isotopic through a compactly supported isotopy of ascending surfaces to $(W,S(U,\nu))$ for some transverse banded link $(U,\nu)$.
\end{proposition}

\begin{proof}
    The proof follows from Proposition~\ref{prop: weinstein_homotopy} and Remark~\ref{rem: ascending_isotopy}.
\end{proof}


\subsection{On complex, quasipositive, and positive ascending surfaces}\label{surface_categories.sec}

We are now ready to tackle the main question of the paper; recall Question~\ref{qn: symp_bridge}.
We answer in the affirmative for positive ascending surfaces, which are a priori a subclass of symplectic surfaces. 
But in some Weinstein domains, these classes are equivalent. 
To discuss what is known in the literature and expected in folklore about the relationship between positive ascending surfaces and symplectic surfaces, we will first set up a few definitions. Recall that a Weinstein domain $W$ induces an open book on its boundary such that the attaching link of 2-handles lies on the page of the open book.

\begin{definition}\label{def: qp}
    A properly embedded symplectic surface in a symplectic filling is called \emph{quasipositive} if its intersection with a compatible boundary open book is a quasipositive braid.
\end{definition}

\begin{definition}\label{def: complex}
    A surface $S$ in a Weinstein domain $W$ is called \emph{complex}
    if it is a complex submanifold with respect to the Stein structure on $W$. It will be referred to as a \emph{complex curve}.
\end{definition}

The following conjectures are reasonable to make, given current evidence --- they can be regarded as a relative version of the Symplectic Isotopy Conjecture (see  \cite{starkston-symplectic_isotopy, siebert_tian}). 
The following conjecture has been made by experts but we did not find a reference.

\begin{conjecture}\label{conj: symp_complex}
    A symplectic surface in a Weinstein domain is symplectically isotopic to a complex curve.
\end{conjecture}

Some weaker versions, stemming from the above discussion, are as follows.

\begin{conjecture}\label{conj: sympl_pos_asc}
    A symplectic surface in a Weinstein domain is symplectically isotopic to a positive ascending surface, up to ambient Weinstein homotopy.
\end{conjecture}

\begin{conjecture}\label{conj: qp_sfc}
 Any pointed quasipositive factorization of a braid in the compatible boundary open book of a Stein domain is the boundary of a complex curve. 
\end{conjecture}

The following result by Hayden \cite{hayden21}, combined with Theorem~\ref{thm: main} and positive resolutions to the above conjectures, would give a positive answer to Question~\ref{qn: symp_bridge} in full generality.

\begin{theorem}\cite{hayden21}\label{thm: hayden}
The following are true:
    \begin{enumerate}
        \item A complex curve in a Stein domain is positive ascending.
        \item A positive ascending surface in a Stein domain is isotopic through positive ascending surfaces to a symplectic surface.
        \item If $K \subset Y$ is the closure of a quasipositive braid with respect to a positive allowable open book, then $K$ bounds a properly embedded quasipositive symplectic surface in a Stein filling of $Y$.
    \end{enumerate}
\end{theorem}

The notion of positive ascending surfaces also generalizes the notion of \emph{positively braided surfaces} in the literature. By the work of Boileau--Orevkov \cite{boileau-orevkov} and Rudolph \cite{rudolph1983algebraic}, progress has been made on Conjectures~\ref{conj: symp_complex} and \ref{conj: qp_sfc} for the standard 4-ball: they are true up to diffeomorphisms. The results of \cite{boileau-orevkov} establish Conjecture~\ref{conj: sympl_pos_asc} in the 4-ball, and in upcoming work \cite{BRW:PANLF} this conjecture is established for Weinstein domains admitting a planar nearly-Lefschetz fibration. Some of these results are also claimed in work in progress of Baykur--Etnyre--Hayden--Hedden--Kawamuro--Van Horn-Morris. 

\begin{remark}\label{rem: braid_isotopies}
    It follows from Theorem~\ref{thm: hayden} above that positive ascending surfaces 
    can be encoded by a quasipositive factorization of the boundary braid, i.e., writing the braid as products of conjugates of the positive generators $\sigma_i \in B_n$:
    \[
    \beta = \Pi_{i=1}^k w_k \sigma_{i_k}w_k^{-1}
     \]
    There are braid moves that preserve the isotopy type, e.g.\ by applying braid relations antisymetrically to $w_k$ and $w_k^{-1}$. The two main nontrivial moves are as follows:
    \begin{enumerate}
        \item Global conjugation $\alpha_1 \dots \alpha_l \sim w \alpha_1 w^{-1} \dots \iff w \alpha_l w^{-1}$
        \item Hurwitz moves $\alpha_1 \dots \alpha_k\alpha_{k+1} \dots \alpha_l \iff \alpha_1 \dots (\alpha_k\alpha_{k+1} \alpha_k^{-1}) \alpha_k \dots \alpha_l$
    \end{enumerate}
    Also, braid stabilization preserves the isotopy type of the positive ascending surface by adding a $\sigma_n$ factor to a factorization in $B_n$, to yield a factorization of the stabilized boundary braid in $B_{n+1}$. Inequivalent factorizations can yield distinct surfaces; see e.g.\ \cite{auroux}.
\end{remark}

\begin{figure}[h]
    \centering
    \includegraphics[width=0.4\linewidth]{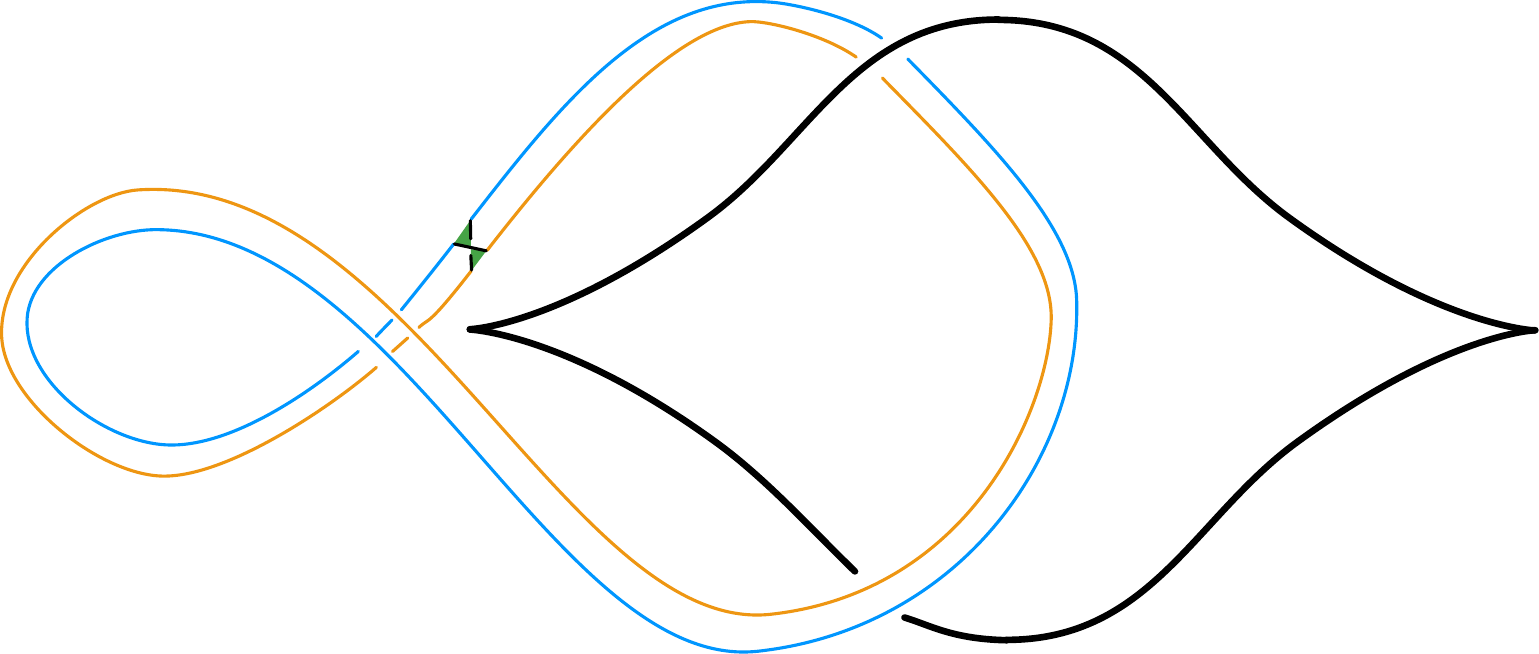}
    \caption{A transverse banded unlink in the Kirby--Weinstein diagram for the cotangent disk bundle of $S^2$.}
    \label{fig:bandedinkwdiagram_1}
\end{figure}

\section{Bisections and multisections with divides for ascending surfaces}

\subsection{Brief overview of the smooth case}
  Trisections of closed 4-manifolds, a decomposition of a 4-manifold into three 4-dimensional 1-handlebodies analogous to Heegaard splittings in dimension 3, were first introduced by Gay and Kirby \cite{gay2016trisecting}.  Meier and Zupan introduced {\it bridge trisections}, an analogous decomposition of knotted surfaces in $S^4$, and later for surfaces in trisected 4-manifolds, in \cite{meier2017bridge,meier2018bridge}. 
  Multisections, or decompositions into possibly more than 3 pieces, of closed 4-manifolds and of surfaces were introduced and studied in \cite{Islambouli-Naylor}.  All of these constructions have the advantage that 4-manifolds and knotted surfaces in them can be represented diagrammatically, by curve-and-arc systems on surfaces.  Multisections allow for greater flexibility in these diagrammatic representations, by potentially allowing one to decrease the genus of the surface required.  
  
\renewcommand{\arraystretch}{1.2}
\begin{table}[h]
  \begin{tabular}{|l|l|}
  \hline
 &Manifold version \\ 
  \hline
  Geometric& Multisection with divides for $(W,\omega)$\\ 
  \hline
  Mapping class & Open book $(F,\phi)$ supporting $(\partial W, \xi)$, admitting factorization \\
  &$\phi=\phi_{n-1,n}\circ\cdots\circ\phi_{2,3}\circ\phi_{1,2}$ such that\\
  &$\phi_{i,i+1}$ induces tight structure on $H_i\cup \overline{H}_{i+1}=\#^{k_i} S^1\times S^2$ \\
  \hline
  Diagrammatic & Multisection diagram with divides\\
  \hline
  \end{tabular}
  \vspace{3mm}
  \caption{Geometric, algebraic, and combinatorial descriptions of decompositions of Weinstein domains \cite{IS-divides}.}
  \label{decompositions.tab} 
\end{table}

\subsection{Multisections with divides for Weinstein domains }
The analogue of multisections for Weinstein domains with boundary was introduced in~\cite{IS-divides}.

\begin{definition}[\cite{IS-divides}, Def.~1.1]\label{multisec.def}
Let $(W, \omega)$ be a symplectic filling for its contact boundary $(M, \xi)$. A \emph{multisection with divides} for $(W,\omega)$ is a decomposition $W = W_1 \cup \cdots \cup W_n$ where
\begin{itemize}
    \item $W_i \cong \natural_{k_i} S^1 \times D^3$
    \item $H_{i+1} := W_i \cap W_{i+1} \cong \natural_g S^1 \times D^2$ for $i = 1, \ldots, n-1$
    \item $\Sigma := W_1 \cap \cdots \cap W_n = \partial H_i$ for all $i$
    \item each $(W_i, \omega|_{W_i})$ is a symplectic filling of $(\partial W_i, \xi_i)\cong (\#^{k_i} S^1 \times S^2,\xi_{std})$
    \item $H_i \cup \overline{H}_{i+1}$ is a contact Heegaard splitting of $(\partial W_i, \xi_i)$, corresponding to the open book $(F,\phi_{i,i+1})$
    \item $H_1 \cup \overline{H}_{n+1}$ is a contact Heegaard splitting of $(\partial W, \xi)$, corresponding to the open book $(F,\phi)$, where $\phi=\phi_{n-1,n}\circ \cdots \circ \phi_{1,2}$
    \item the contact structure on each $H_i$ induces the same dividing set on $\Sigma$
\end{itemize}
\end{definition}

\subsubsection{Multisection diagrams with divides}

As in the smooth case, multisections of Weinstein domains allow one to represent such manifolds diagrammatically via a collection of cut systems on a surface, called a {\it multisection diagram with divides} \cite{IS-divides}.  These diagrams can be viewed as a 4-dimensional analogue of contact Heegaard diagrams.

\begin{definition}[\cite{IS-divides}, Def.~3.2]
    A \emph{multisection diagram with divides} is a tuple $(\Sigma, \alpha_1,\dots, \alpha_n, \Gamma)$ such that $\alpha_i=\{\alpha_{i,j}\}$ are cut systems for $\Sigma$, and $\Gamma$ is a multicurve cutting $\Sigma$ into two homeomorphic surfaces with boundary and such that the intersection of the $\alpha_{i,j}$ with each half of $\Sigma\setminus \Gamma$ is an arc system.  In addition, $(\Sigma, \alpha_i,\alpha_{i+1},\Gamma)$ is a contact Heegaard diagram of the tight contact structure on $\#^{k_i}S^1\times S^2$ for all $i\in \{1,n-1\}$.
\end{definition}

A bisection diagram with divides for the cotangent disk bundle to $S^2$ is given in Figure~\ref{fig:bisectiondiagram}.
\begin{figure}[h]
    \centering
    \includegraphics[width=0.45\linewidth]{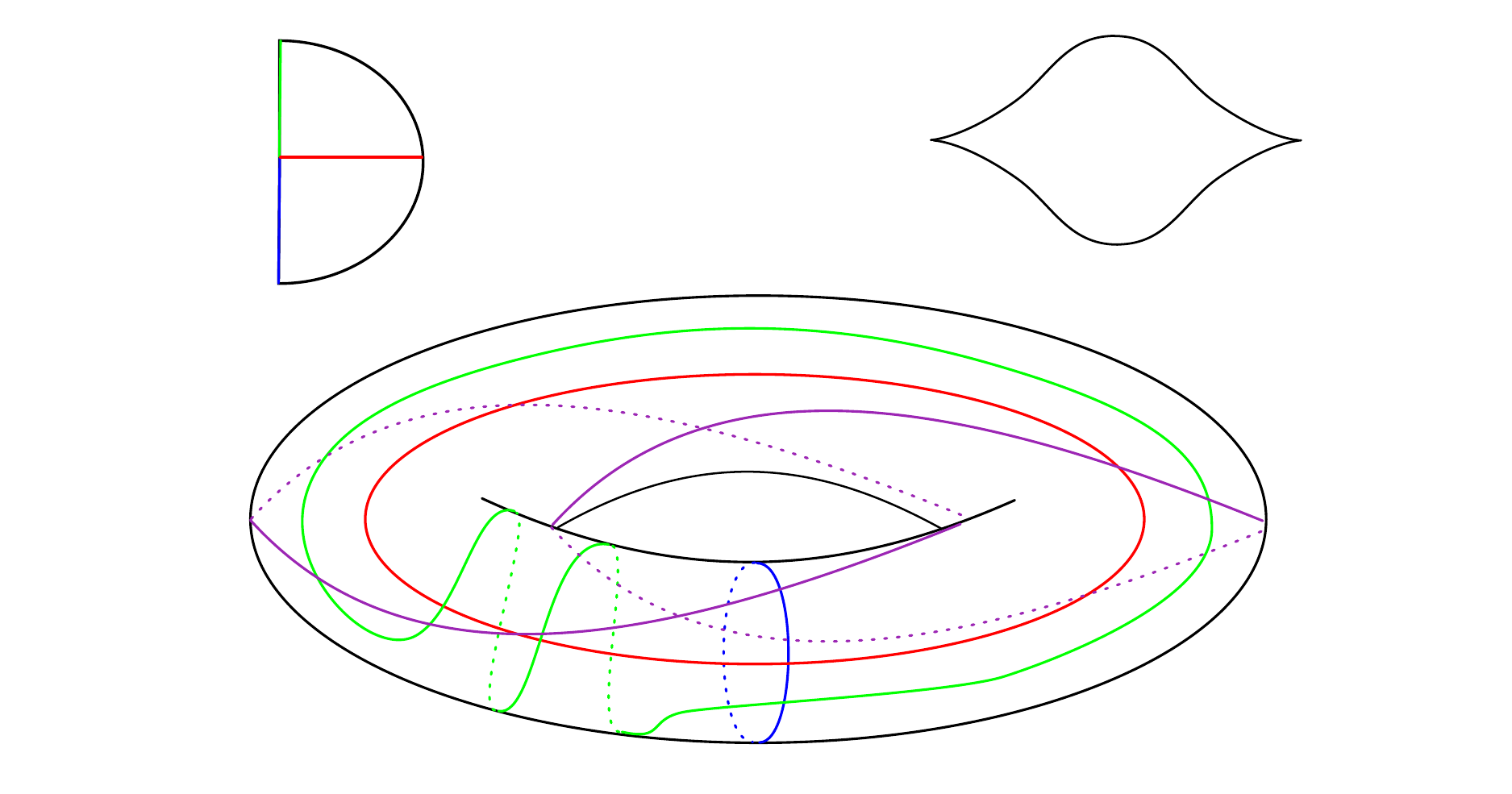}
    \caption{A bisection diagram with divides for the cotangent disk bundle of $S^2$.}
    \label{fig:bisectiondiagram}
\end{figure}

\subsubsection{Bisections with divides for ascending surfaces} We now introduce our main construction, a bridge decomposition of positive ascending surfaces analogous to 
the constructions of Meier--Zupan in the smooth case, namely a bridge multisection with divides.  In the smooth construction, bridge-multisected surfaces are decomposed into a union of trivial disk systems.

\begin{table}[h]
  \begin{tabular}{|l|l|}
  \hline
  &Surface version\\
  \hline
  Geometric&  Multisection with divides for a surface $S\subset (W,\omega)$ (Def.~\ref{bridge_multisection_divides.def})
  \\
  \hline
  Mapping class& Pointed open book $(F,\hat\phi)$ for the closed braid $L=\partial S$ 
  \\ & $\hat\phi=\hat \phi_{n-1,n}\circ \cdots \circ\hat \phi_{1,2}$ such that\\
  &Underlying unpointed monodromies $\hat\phi_{i,i+1}$ induce tight contact structure\\
  &$\hat\phi_{i,i+1}$ corresponds to a \textmaxsl unlink in $\#^{k_i} S^1\times S^2$ (Def.~\ref{def:PMC})\\
  \hline
  Diagrammatic& Shadow diagram with divides (Def.~\ref{shadowdiagram.def}),
  \\ & or braided tri-plane diagram in the case $W=B^4$ (Def.~\ref{def: braided_tri-plane})\\
  \hline
  \end{tabular}
  \vspace{3mm}
  \caption{Geometric, algebraic, and combinatorial descriptions of decompositions of positive, ascending surfaces.}
\label{surfacedecompositions.tab}
\end{table}

\begin{definition} A \emph{trivial symplectic $c$-disk system} is a collection of $c \in \N$ properly embedded symplectic disks $D_i$ in a Weinstein 1-handlebody $V$ as in Proposition~\ref{uniquedisks.prop}, where each $\partial D_i$ is a maximal self-linking unknot in $\partial V$. 
\end{definition}

It follows from Proposition~\ref{uniquedisks.prop} that the $\partial D_i$ are simultaneously smoothly isotopic into $\partial V$, i.e., are smoothly trivial disk systems.

\begin{definition}\label{bridge_multisection_divides.def}
    Let $S$ be a 
    surface properly embedded in a Weinstein domain $W$, where $W=W_1\cup\dots \cup W_n$ is a multisection with divides as in Definition \ref{multisec.def}.  A $(b; c_1,\dots, c_n)$ \emph{bridge multisection with divides} is a decomposition
    $$(W,S)=(W_1,\mathcal{D}_1)\cup \dots \cup(W_n,\mathcal{D}_n)$$ such that for each $1\leq i \leq n$, with indices taken mod $n$, $\mathcal{D}_i \coloneqq S\cap W_i$ is a trivial symplectic $c_i$-disk system, $(H_i,\tau_i)=(W_i,\mathcal{D}_{i})\cap (W_{i+1},\mathcal{D}_{i+1})$ is a $b$-strand braid in the half-open book structure on $H_i$, such that consecutive unions $(H_i,\tau_i)\cup (\bar H_{i+1},\bar \tau_{i+1})$ form a braided \textmaxsl unlink in the corresponding open book, and $(\Sigma,\{p_1,\dots,p_{2b}\})=\bigcap_{i=1}^n (W_i,\mathcal{D}_i)$ is a surface  with $2b$ marked points.  The {\it pointed monodromy of the $i^{\text{th}}$ sector} is the pointed monodromy of the mapping class corresponding to the closed braid $\tau_i \cup\overline{\tau}_{i+1}$ in the open book corresponding to $H_i\cup \overline{H}_{i+1}$. 
\end{definition}
We also say such a surface is in {\it bridge position} with respect to a multisection with divides for $W$.

A {\it bisection} is a multisection with $n=2$ 4-dimensional sectors.
In the smooth case, a multisection of a 4-manifold is determined uniquely by the union of its 3-dimensional handlebodies, called its {spine}. 
In our setting, a spine is defined as follows.
\begin{definition}\label{def: spine_surface_with_divides}
A {\it spine} is a tuple $(H_1,\cdots, H_n;\tau_1,\cdots, \tau_n)$, of contact handlebodies $H_i$ and braids $\tau_i$ in the corresponding half-open book structure on $H_i$, such that 
\begin{itemize}
    \item each consecutive union $H_i\cup \bar H_{i-1}$, $i=2,\dots, n$, is a  connected sum of $S^1\times S^2$'s, and the induced contact structure is its standard tight contact structure;
    \item the induced contact structure on $H_1\cup \bar H_n$ is also tight;
    \item $(H_i,\tau_i)\cup (\bar H_{i-1},\bar \tau_{i-1})$ forms a braided \textmaxsl unlink in a the open book structure on $H_i\cup \bar H_{i-1}$ determined by the half-open book structures on $H_i$ and $H_{i-1}$; 
    \item $(\Sigma,\{p_1,\dots,p_{2b}\})=\cap_{i=1}^n (H_i,\tau_i)$ is a surface  with $2b$ marked points; and
    \item the contact structure on each $H_i$ induces the same dividing set on $\Sigma$.
\end{itemize}
\end{definition}

\begin{proposition}\label{prop: spine_determines_symplectic surface}
    Let $(H_1,\dots, H_n;\tau_1,\dots \tau_n)$ be a spine as in Definition~\ref{def: spine_surface_with_divides}. The spine determines a well-defined bridge multisection $(W,S)$ of a symplectic surface $S$ in a Weinstein domain $W$. 
\end{proposition}

\begin{proof}
    It suffices to determine the 4-dimensional handlebodies $W_i$ and the symplectic disk systems $\mathcal{D}_i$ in them. This is true since each $(H_i\cup \bar H_{i-1},\xi_i)=(\#_{k_i}S^1 \times S^2,\xi_{std})$ has a unique Stein filling up to deformation~\cite{mcduff_rational-ruled}; see also Remark 1.2 of \cite{IS-divides}. Also, each \textmaxsl unlink $\tau_i\cup \bar{\tau}_{i-1}$ bounds a unique symplectic disk system by Proposition~\ref{uniquedisks.prop}.
\end{proof}

\subsection{Existence and banded bridge position}

We now discuss how to find bridge multisections with divides for positive ascending surfaces. The input will be a transverse banded unlink $(U,\nu)$ inside a Weinstein handlebody diagram $(\Lambda,k)$ for $W$. We first review how bisections of $W$ can be built from $(\Lambda,k)$ (Rem.~\ref{rem:KWtobisection}), and then we discuss what embeddings of the surface data $(U,\nu)$ will determine a multisection with divides (Def.~\ref{def:banded bridge position}).

\subsubsection{Bisections from Kirby--Weinstein diagrams~\cite{IS-divides}} \label{rem:KWtobisection}
Let $(W,\omega)$ be a Weinstein domain with a fixed Kirby--Weinstein diagram $(\Lambda,k)$. Let $W_1=\natural^{k_1} S^1 \times D^3$ be the union of the 0- and 1-handles of $(\Lambda,k)$, with $k_1=k$, and symplectic structure given by the filling of $(\#^{k_1} S^1\times S^2,\xi_{std})$. Let $\partial W_1= H_1\cup_\Sigma \overline{H}_2$ be a contact Heegaard splitting of $(\#^{k_1} S^1\times S^2,\xi_{std})$ so that $\Lambda$ is a subset of the Legendrian core of $H_2$, and so that the framing on each component $\Lambda_i$ of $\Lambda$ induced by the splitting surface $\Sigma$ is $tb(\Lambda_i)$. 
As discussed in Section~\ref{section:contact heegaard}, this contact Heegaard splitting induces an open book decomposition $(\Gamma,\pi)$ for $\partial W_1$. Finally, let $H_3=H_2[\Lambda]$, the result of Legendrian surgery (framing $tb(\Lambda_i)-1$) on $H_2$ along the link $\Lambda$, and let $W_2$ be a collar of $H_2 $ together with the 2-handles attached along $\Lambda$.  Observe that $\partial W_2=H_2\cup_\Sigma\overline{H_3}$, and $\partial W=H_1\cup_\Sigma \overline{H}_3$ are both contact Heegaard splittings with Heegaard surface $\Sigma$ and dividing set the same set of curves $\Gamma\subset \Sigma$. Under these conditions, Islambouli and Starkston showed in \cite[Thm.~3.1]{IS-divides} that both $W_1$ and $W_2$ are the unique Weinstein fillings of $\#^{k_j}S^1\times S^2$ for $i=1,2$. Hence, $W_1\cup W_2$ is a bisection of $W$ with divides. Bisections built this way will be denoted by $\Tcal(\Lambda,k,\Sigma)$.

\subsubsection{Positive bands and banded bridge position}

We return to our transverse banded unlink $(U,\nu)$ in the Kirby--Weinstein diagram $(\Lambda,k)$. If $U$ is braided with respect to an open book decomposition $(\Gamma,\pi)$ with pages $F_\theta$, then transverse banded unlinks $(U,\nu)$ can be described by a pointed mapping class group element $F_0\rightarrow F_1$ representing an unlink, together with positive bands whose cores lie in pages of the open book.  To build a multisection with divides for $S(U,\nu)$, this braiding will be done with respect to the open book decomposition in Section~\ref{rem:KWtobisection}.
\begin{example}\label{eg: framedband}
    Consider $(M, \xi)$ supported by an open book $(\Gamma, \pi)$ with page $F_{\theta}$ for $\theta \in S^1$, and a transverse link $L$ which is braided with respect to $(\Gamma,\pi)$. Define a collection of annuli $\mathcal{A}$ with as many components as $|L|$, in $(M,\xi)$, by thickening the points in $L \cap F_{\theta}$ to intervals $[-1,1]$. The boundary of these annuli are two links $L_{-1}$ and $L_1$ that are both isotopic to $L$. Fix a page $F_{\theta_0}$, and let $\alpha$ be an arc in $F_{\theta_0}$ connecting two points in $L \cap F_{\theta_0}$ that live near the binding. Then, modify $\mathcal{A}$ by adding a half-twisted band along $\alpha$ as in Figure~\ref{fig:braidedbands};
    the resulting surface in $(M, \xi)$ is a collection of annuli and a single pair of pants. Call this modified surface $S.$ Note that the open book foliation on the pair of pants component of $S$ has one hyperbolic singularity 
    as in Fig 4.(b) of \cite{hayden21}.
\end{example}
\begin{figure}[h]
\centering
\includegraphics[width=8cm]{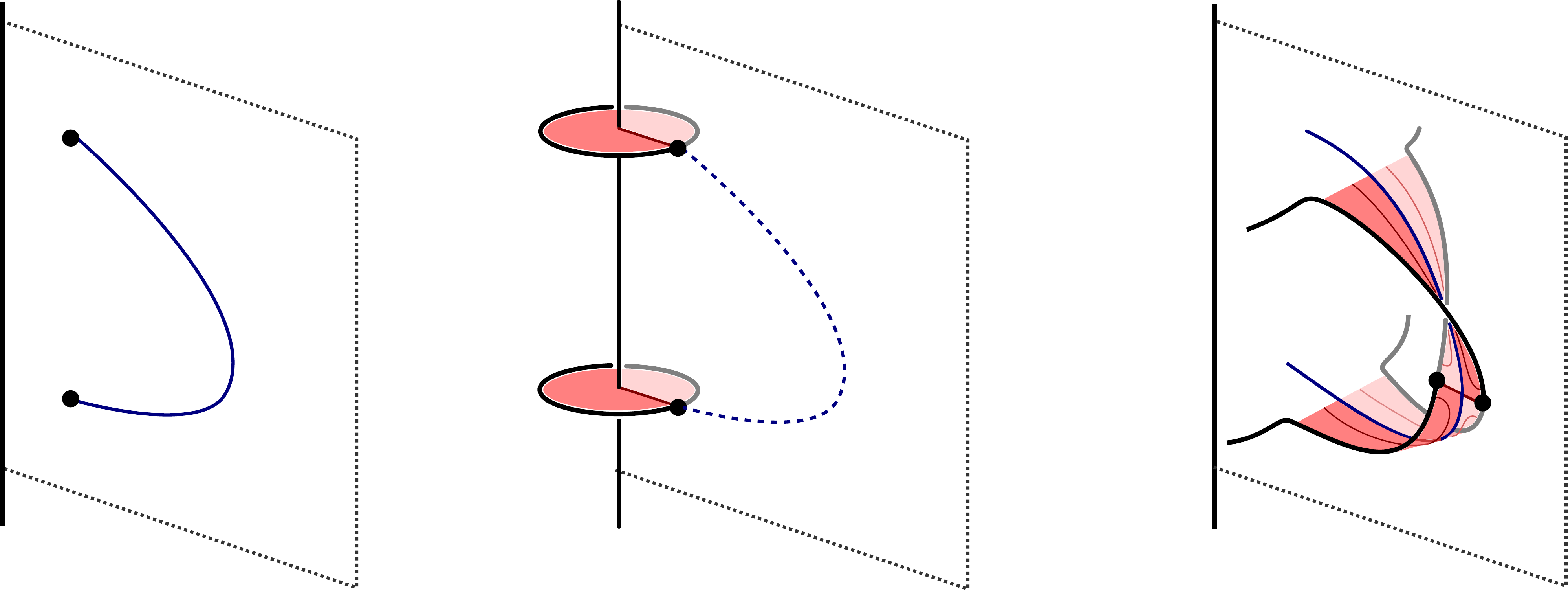}
\caption{(1/2)-framed braided bands are determined by arcs embedded in a page of the open book decomposition. We thank Kyle Hayden for allowing us to import this figure from \cite{hayden21}.}
\label{fig:braidedbands}
\end{figure}

\begin{definition}\label{def: framedband}
    A positive band for which the image of the surface $S$ when projected to $M$ is as in Example~\ref{eg: framedband} is called a \emph{(1/2)-framed band determined by the arc $\alpha$}. 
\end{definition}

\begin{definition}\label{def:braid_stab}
Let $L$ be a braid in the the open book $(\Gamma, \pi)$ and let $U$ be a small meridian of the binding $\Gamma$. A \emph{positive Markov stabilization} of $L$ is the new braid obtained by adding a $(1/2)$-framed braided band to the link $L\cup U$ along an arc with endpoints in $L$ and $U$. 
\end{definition}
Our definition of Markov stabilization is consistent with Pavelescu's~\cite{pavelescu2012braiding}. It changes the pointed open book associated to $L$ by adding a half-twist between a puncture of $L$ and a new puncture; see \cite{hayden21} for details.

\begin{definition}\label{def:banded bridge position}
A transverse banded link $(L,\nu)$ is said to be in \emph{banded bridge position} with respect to an open book $(\Gamma,\pi)$ if

\begin{enumerate}
    \item[(BB-1)] \namedlabel{item:bb1}{(BB-1)} $L$ is braided with respect to $(\Gamma,\pi)$, with corresponding transverse bridge splitting $\tau_i=L\cap H_i$ in the corresponding half-open book decomposition of $(M,\xi)$;
    
    \item[(BB-2)] \namedlabel{item:bb2}{(BB-2)} each band of $\nu$ is a $1/2$-framed braided band as in Definition \ref{def: framedband};
    
    \item[(BB-3)] \namedlabel{item:bb3}{(BB-3)} the cores of $\nu$ are embedded in the same page $F_{0}$, with $\partial H_1=F_{0}\cup \overline{F_{1/2}}$, such that they have pairwise disjoint interiors and the cores of $\nu$ form an acyclic 1-complex in $F_{0}$; and
    
    \item[(BB-4)] \namedlabel{item:bb4}{(BB-4)} there is a choice of shadows for $\tau_2$ on $F_0\cup \overline{F_{1/2}}$, crossing $\partial F_0$ once transversely, that are disjoint from the cores of $\nu$.
    
\end{enumerate}
In this case, we say that the cores of $\nu$ are dual to the braid $L$. We refer to the elements of $\nu$ as {\it dual bands}. 
\end{definition}

Notice that the definition above is for links rather than unlinks to allow the possibility of working with multisections with more than two 4-dimensional sectors. 
Intuitively, we want the bands from \ref{item:bb2} to be described by half-twists along arcs in a fixed page, and the union of such core arcs (from \ref{item:bb3}) to form a forest with vertices being the punctures of the page. 
In practice, to test if \ref{item:bb3} holds, we will slide the bands to lie near each other, in a crossingless neighborhood of the braid $L$. See how we project the cores of the bands to one page in the left panels of Figures~\ref{fig:fig_qp_trefoil} and \ref{fig:BVM_5}.

Proposition~\ref{prop:banded_gives_bisection} states that transverse banded unlinks in banded bridge position induce multisections with divides for their corresponding ascending surfaces. In what follows, we will explain how to find the spine of a bridge bisection with divides from such a transverse banded unlink. For proofs, we refer the reader to Section~\ref{sec:proofs}.

Consider a bisection with divides $W=W_1\cup W_2$ 
as in Section~\ref{rem:KWtobisection} and an ascending surface $S(U,\nu)$ realizing the banded unlink $(U,\nu)$ (as constructed in Proposition~\ref{realizing-symplectic-surface.prop}) inside of $W$. Suppose that $(U,\nu)$ is in banded bridge position, and suppose that the bands in $\nu$ are inside the handlebody $H_2$, near the boundary of $H_2$. 
Let $\tau_i=U\cap H_i$ be the braided arcs in the half-open books $H_i$ for $i=1,2$. 
As the attaching region of the 2-handles $\Lambda$ is a subset of the core of $H_2$, the surgered tangle $\tau_2[\nu]$ can be regarded as a transverse arc in the surgered handlebody $H_2[\Lambda]=H_3$. In turn, $H_2\cup\overline{H}_3$ comes with an open book decomposition with binding $\Gamma$ for which $\tau_2\cup \overline{\tau_2[\nu]}$ is a braid. The proof of the following lemma is contained in that of Proposition~\ref{prop:banded_gives_bisection}.

\begin{lemma}\label{lem:spine_from_bridgeposition}
Given $S(U,\nu)\subset W(\Lambda,k)$ as above, the pair admits a bridge bisection with divides with spine $\left(H_1, H_2, H_2[\Lambda]; \tau_1, \tau_2, \tau_2[\nu]\right)$. 
\end{lemma}

\section{Diagrams of multisections with divides for ascending surfaces}

Multisections of positive ascending surfaces in Weinstein domains can be encoded diagrammatically. In this section, we introduce three approaches to describing a bridge bisection with divides: \textit{braided tri-planes}, \textit{shadow diagrams}, and \textit{multisections of pointed monodromies}. 
These are the symplectic analogues of the tri-plane diagrams and shadow diagrams for smooth surfaces in \cite{meier2018bridge}. In Example~\ref{Baykur_VHM_examples}, we begin with a quasipositive factorization of a braid, which determines a positive ascending surface in the 4-ball, and derive a braided tri-plane diagram. 
In Example \ref{2_compo_unlink_one_band_example}, we start with a Kirby--Weinstein picture of a positive ascending surface and derive the second and third descriptions. 

\subsection{\textmaxsl braided tri-plane diagrams and examples}\label{sec:braided_tri-plane_examples}
A bisected positive ascending surface in $B^4$ can be represented by a tri-plane diagram, as in \cite{meier2017bridge}, with additional restrictions on the tangles. There is a natural generalization to diagrams for multisected surfaces in $B^4$.  Below, $\overline{\tau}$ denotes the mirror image of the tangle $\tau$, i.e., the inverse, when considered as an element of the braid group on $b$ strands.

\begin{definition}
\label{def: braided_tri-plane}
A $(b; c_1,c_2)$-\emph{\textmaxsl braided tri-plane diagram} is a triple $(\tau_1,\tau_2,\tau_3)$ of braids such that $\tau_1\cup \overline{\tau}_3=\partial S$, $\tau_1\cup \overline \tau_2$ is a $c_2$-component \textmaxsl unlink, and $\tau_2\cup \overline{\tau}_3$ is a $c_3$-component \textmaxsl unlink.    
\end{definition}

\begin{remark}
    Braided tri-plane diagrams were introduced in \cite[Def 1.2]{aranda25}. Because our surfaces have boundary, we only require that two of the three pairwise unions of tangles are unlinks.  Aside from that fact, our \textmaxsl braided tri-plane diagrams satisfy criteria 1) and 2) of \cite[Definition 1.2]{aranda25}. In particular, a \textmaxsl braided  unlink is a positive stabilization of the trivial braid. Because all of the braided tri-plane diagrams in this paper are max-\textit{sl}, we will simply say ``braided tri-plane diagram" from now on.
\end{remark}

To obtain a braided tri-plane presentation for $S$ from a banded unlink presentation $(U,\nu)$ for $S$, one first isotopes 
$(U,\nu)$ into banded bridge position (Def.~\ref{def:banded bridge position}) using a modification of Procedure 5.8 of \cite{aranda25} (see the proof of Proposition~\ref{prop:bandedunlink_implies_bandedbridgeposition} for details). The corresponding bridge decomposition $U=\tau_1\cup \overline{\tau}_2$ gives the first two tangles; the third tangle $\tau_3=\tau_2[\nu]$ is obtained from $\tau_2$ by resolving the (positive) bands.  A banded diagram and braided tri-plane diagram for an ascending surface bounded by the positive trefoil is given in Figure~\ref{fig:fig_qp_trefoil}.

\begin{example}\label{Baykur_VHM_examples}
Work of Baykur and Van Horn-Morris \cite{baykur_VHM_infinite} gives examples of infinitely many pairwise distinct fillings of the same quasipositive braid; see Figure~\ref{fig:BVM_12}(A). These surfaces are described as braided banded presentations $(L,\nu)$ where $\nu=\{v_1,\dots, v_4\}$ and $v_1$ has $n$ twists. In Figure~\ref{fig:BVM_12}(B), we only draw the cores of the bands $v_1,\dots, v_4$ and project $v_1$ and $v_2$ onto the same disk page; note that these projections may intersect as $v_1$ is obtained by $n$ half-twists. In such a disk, we find an arc $\alpha$ disjoint from the cores of $v_1$ and $v_2$ that connects an endpoint of $v_2$ with the boundary of the disk page. The red stabilization arc is a push out of the endpoint union of $\alpha$ with the core of $v_2$. 
After stabilizing $L$ in Figure~\ref{fig:BVM_12}(C), we slide the bands $v_2$ and $v_3$ down the new crossings to the bands in Figure~\ref{fig:BVM_12}(D). 
(In general, this last step may not be an easy task as the stabilizing red arc may be complicated.) Note that sliding an arc past a stabilization along red arc is the same as a half-twist map along the same arc; see Figure~\ref{fig:BVM_12}(E). Parts (F) and (G) of Figure~\ref{fig:BVM_12} are the image under the half-twist corresponding to sliding past bands $v_2$ and $v_3$, respectively. 
Now, as $v_1$, $v_2$, and $v_3$ lie in a neighborhood of the same disk page, we can project them onto the disk and observe that the core of $v_3$ crosses the core of $v_1$ (for large $n$). \\
In Figure~\ref{fig:BVM_3}(A), we find a new red stabilization arc to get a new braid representing $L$. Observe that $v_3$ slides past the new crossings as in Figure~\ref{fig:BVM_3}(C); again, one can perform a half-twist to verify this claim. To end, we slide the band $v_4$ up so that it appears below $v_1$ and projects as in Figure~\ref{fig:BVM_4}(A). 
After one last stabilization and arc slide, we obtain a banded bridge position as in Figure~\ref{fig:BVM_4}(C). The respective braided tri-plane diagram is drawn in Figure~\ref{fig:BVM_5}. 

\end{example}

\begin{figure}[h]
    \centering    
    \includegraphics[width=0.9\linewidth]{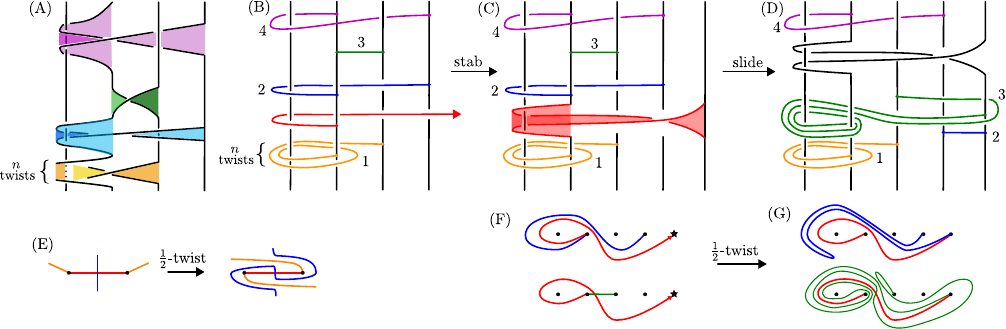}
    \caption{Ribbon surfaces with the isotopic boundary (1/4). (A) Braided banded presentations of such surfaces; in (B), we only drew the cores of the bands. (B)-(D) show modifications to the banded descriptions explained in Example~\ref{Baykur_VHM_examples}.}
    \label{fig:BVM_12}
\end{figure}

\begin{figure}[h]
    \centering
    \includegraphics[width=0.7\linewidth]{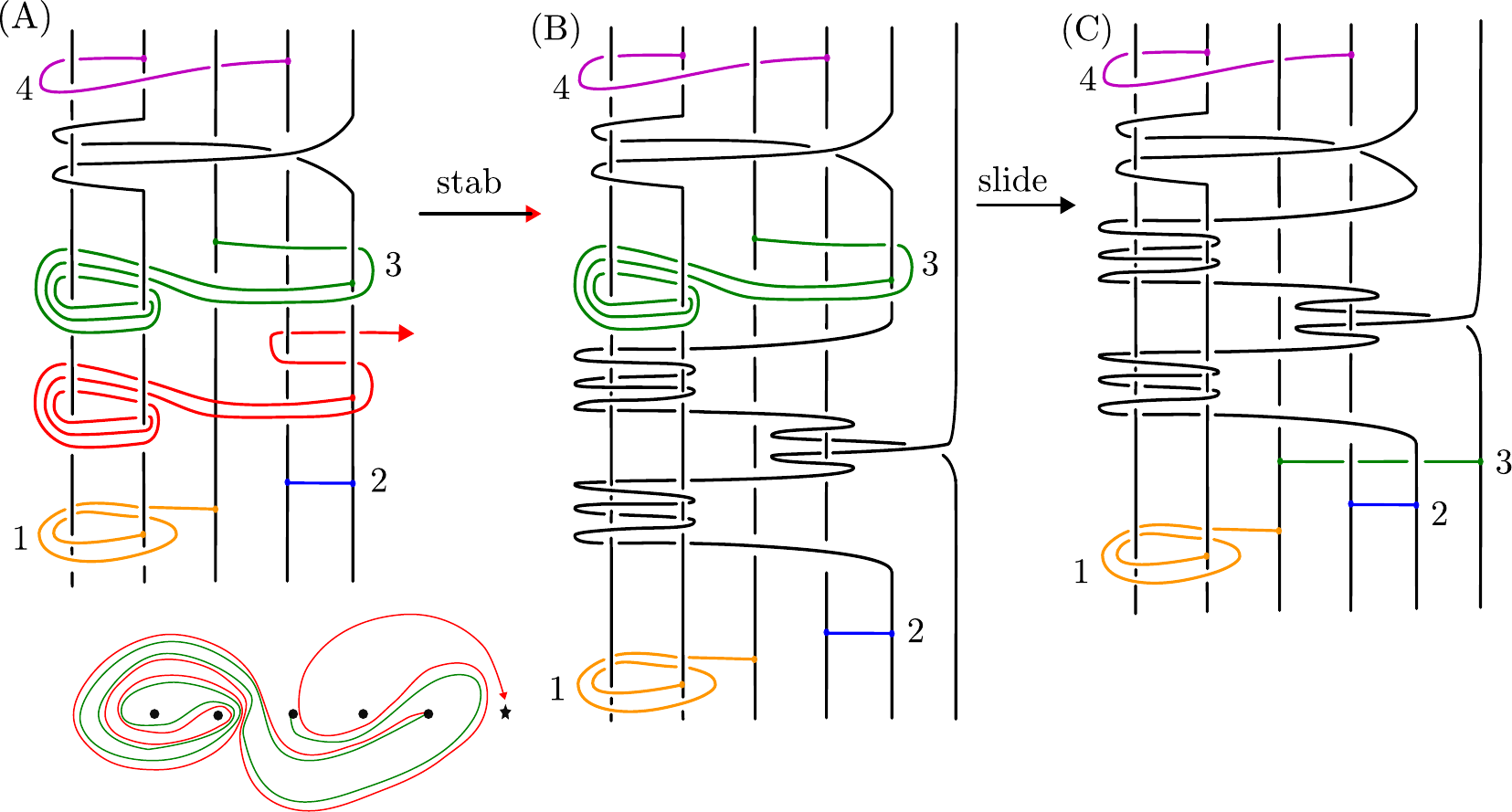}
    \caption{Ribbon surfaces with the isotopic boundary (2/4). (A)-(C) show modifications to the banded descriptions explained in Example~\ref{Baykur_VHM_examples}.}
    \label{fig:BVM_3}
\end{figure}

\begin{figure}[h]
    \centering   
    \includegraphics[width=0.7\linewidth]{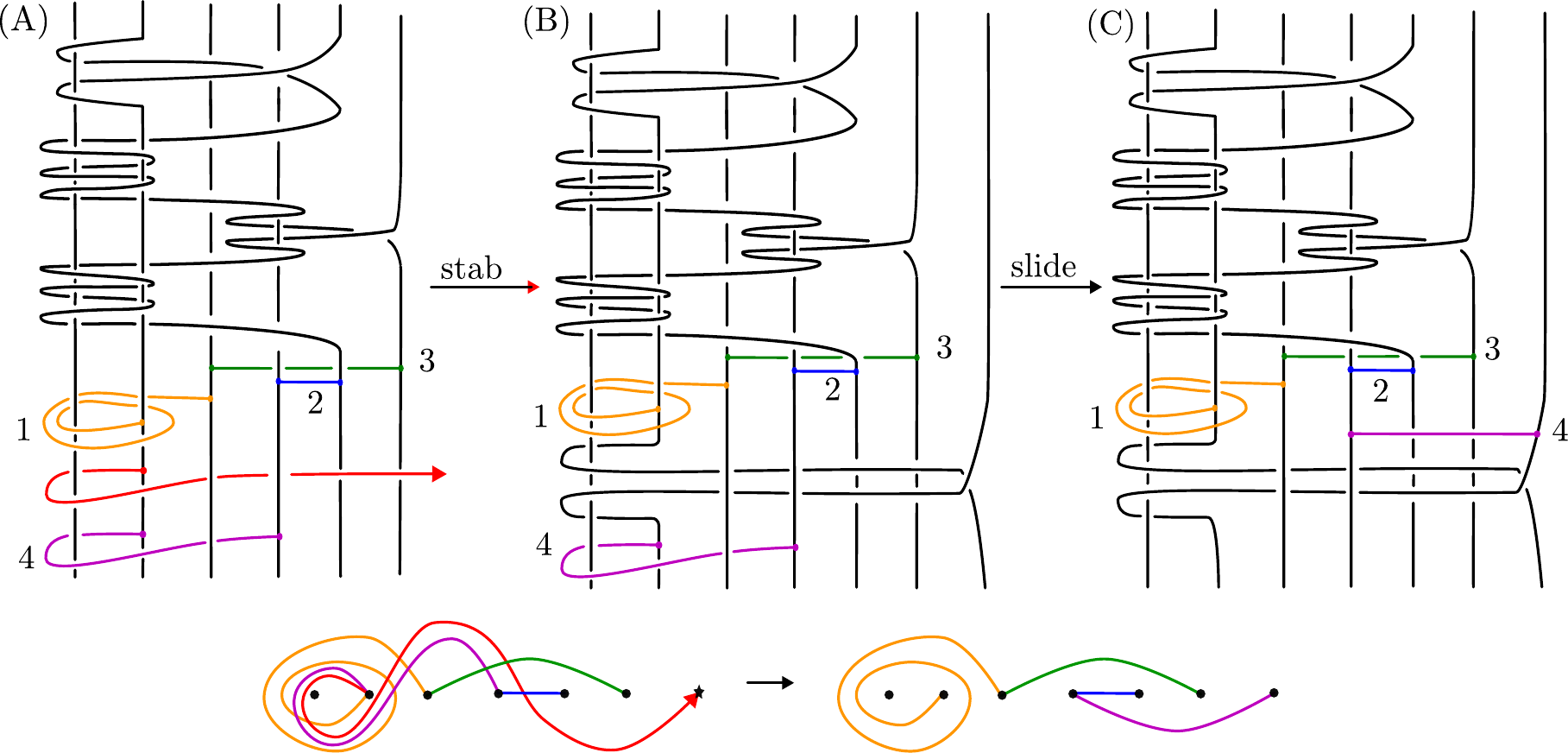}
    \caption{Ribbon surfaces with the isotopic boundary (3/4). (A)-(C) show modifications to the banded descriptions explained in Example~\ref{Baykur_VHM_examples}.}
    \label{fig:BVM_4}
\end{figure}

\begin{figure}[h]
    \centering
    \includegraphics[width=0.8  \linewidth]{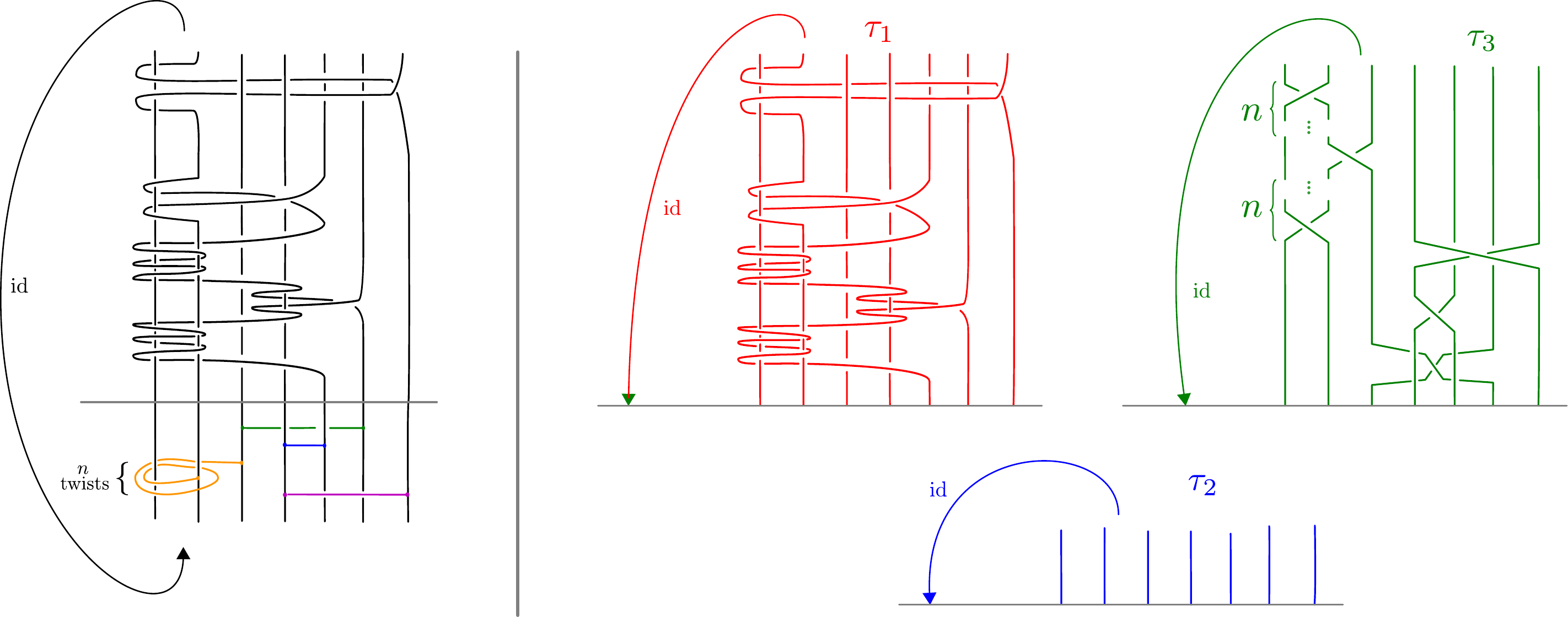}
    \caption{Ribbon surfaces with the isotopic boundary from Example~\ref{Baykur_VHM_examples} (4/4). (Left) Banded bridge position and (right) braided tri-plane diagram for these surfaces.}
    \label{fig:BVM_5}
\end{figure}

\subsection{Shadow diagrams and pointed monodromies} We are ready to introduce two approaches to describing a bridge bisection with divides: shadow diagrams and multisections of pointed monodromies. 

\begin{definition}[SDD]\label{shadowdiagram.def}
    A \emph{shadow diagram with divides} is a multisection diagram with divides $(\Sigma,\alpha_1,\dots,\alpha_n, \Gamma)$ together with a collection $(\tau_1',\dots,\tau_n')$ of $n$ $b$-tuples $\tau_i'$ of disjoint arcs with endpoints $\{p_1,\dots,p_{2b}\}\subset \Sigma$.  Each arc in each $\tau_i'$ intersects $\Gamma$ geometrically once, and consecutive pairs of arc $b$-tuples $(\tau_i',\tau_{i+1}')$ are carried to one another by a pointed mapping class group element $\hat\phi_{i,i+1}$ representing a \textmaxsl unlink in the contact Heegaard splitting $(\Sigma,\alpha_i,\alpha_{i+1})$ of the tight contact structure on $\#^{k_i} S^1\times S^2$, for $i\in\{1,\dots,n-1\}$. 
\end{definition}

\begin{proposition}\label{prop:SDD}
    Let $S$ be a symplectic surface in bridge position with respect to a multisection with divides structure on $(W,\omega)$.  Then there exists a shadow diagram with divides for $S$. Conversely, a shadow diagram with divides determines a bridge-bisected symplectic surface in the Weinstein domain corresponding to the underlying multisection diagram with divides.
\end{proposition}
\begin{proof}
Following the discussion at the end of Section~\ref{subsec: braids}, braids in half-open books can be described by shadow diagrams. Thus, given a bridge bisection with divides, its spine determines a shadow diagram with divides. Conversely, shadows in a half-open book describe unique isotopy classes of braids in half-open books rel boundary. In particular, a shadow diagram describes a well-defined spine as in Definition~\ref{def: spine_surface_with_divides}. The second statement follows from Proposition~\ref{prop: spine_determines_symplectic surface}.
\end{proof}

\begin{remark} If the pointed mapping class $\hat\phi_{i,i+1}$ in Definition~\ref{shadowdiagram.def} is PA-quasipositive, then it represents a \textmaxsl unlink in the tight contact structure on $\#^{k_i} S^1\times S^2$.  The examples considered in this paper all arise from PA-quasipositive mapping classes. 
\end{remark}

The arcs $\tau_i'$, also called {\it shadow arcs}, are the images of the trivial tangles $\tau_i$ in a multisection of $S$ under an isotopy to the bridge surface $\Sigma.$

\begin{definition}[PMC]\label{def:PMC}
    Let $\hat\phi:(F,P)\rightarrow (F,P)$ be a pointed mapping class representing a transverse link in an open book for a tight contact structure on a 3-manifold $M$. A \emph{multisection of the pointed mapping class $\hat\phi$} (bisection if $n=2$) is a factorization $\hat\phi=\hat\phi_{n-1,n}\circ,\dots,\circ\hat\phi_{1,2}$ such that each pointed mapping class $\hat\phi_{i,i+1}$ represents a \textmaxsl unlink in the tight contact structure on $\#^{k_i}S^1\times S^2$ for some $k_i\in \mathbb{N}$.
\end{definition}

\begin{proposition}\label{prop:PMC}
    Let $S$ be a symplectic surface in bridge position with respect to a multisection-with-divides structure on $(W,\omega)$. Let $\hat\phi:(F,P)\rightarrow (F,P)$ be a pointed mapping class representing the transverse link $\partial S$ in an open book for the tight contact structure $\partial W$. Then the pointed mapping class $\hat\phi$ admits a multisection. Conversely, a multisection of a pointed mapping class representing $\partial S$ determines a bridge multisected symplectic surface in the Weinstein domain corresponding to the underlying multisection diagram with divides.
\end{proposition}
\begin{proof}
    Under these assumptions, pointed mapping class factorizations correspond to shadow diagrams with divides. Hence, the proof is analogous to that of Proposition~\ref{prop:SDD}.
\end{proof}

\begin{example}\label{2_compo_unlink_one_band_example}
 Consider the 2-component transverse unlink $(L,\nu)$ with a single positive band, such that the components link $\Lambda$ once positively, shown in Figure~\ref{fig:bandedinkwdiagram_1}. We now show how to describe this surface by a shadow diagram with divides (Fig.~\ref{fig:shadowdiagram}) and by a bridge bisection of a PA-quasipositive mapping class (Fig.~\ref{fig:pointedmonodromy}). 
 
 We use the construction of Avdek \cite{avdek2013contact} to draw a page $F$ of the open book $(\Gamma,\pi)$ in $(S^3,\xi_{std})$ whose core is $\Lambda$, as shown in Figure~\ref{fig:braidedbandedinkwdiagram}.  
 We first use a transverse isotopy to put $(L,\nu)$ in banded bridge position with respect to $(\Gamma,\pi)$.  
 This is possible by Proposition~\ref{prop:bandedunlink_implies_bandedbridgeposition}.  
 We certify that banded bridge position has been achieved in the front diagram as follows:
 \begin{enumerate}
 \item The isotopy of $L$ is achieved through transverse front diagram moves in the complement of $\Lambda$, such at all times the band remains positive.
 \item Each (transverse) intersection of $L$ with each page $F_t$ has the same sign.  In this example, we can see this in the diagram because the link lies in a tubular neighborhood of the binding.
 \end{enumerate}
The banded unlink diagram in Figure~\ref{fig:bandedinkwdiagram_1} is shown in braided position with respect to $F$ in Figure~\ref{fig:braidedbandedinkwdiagram}.

\begin{figure}[h]
    \centering
    \includegraphics[width=0.5\linewidth]{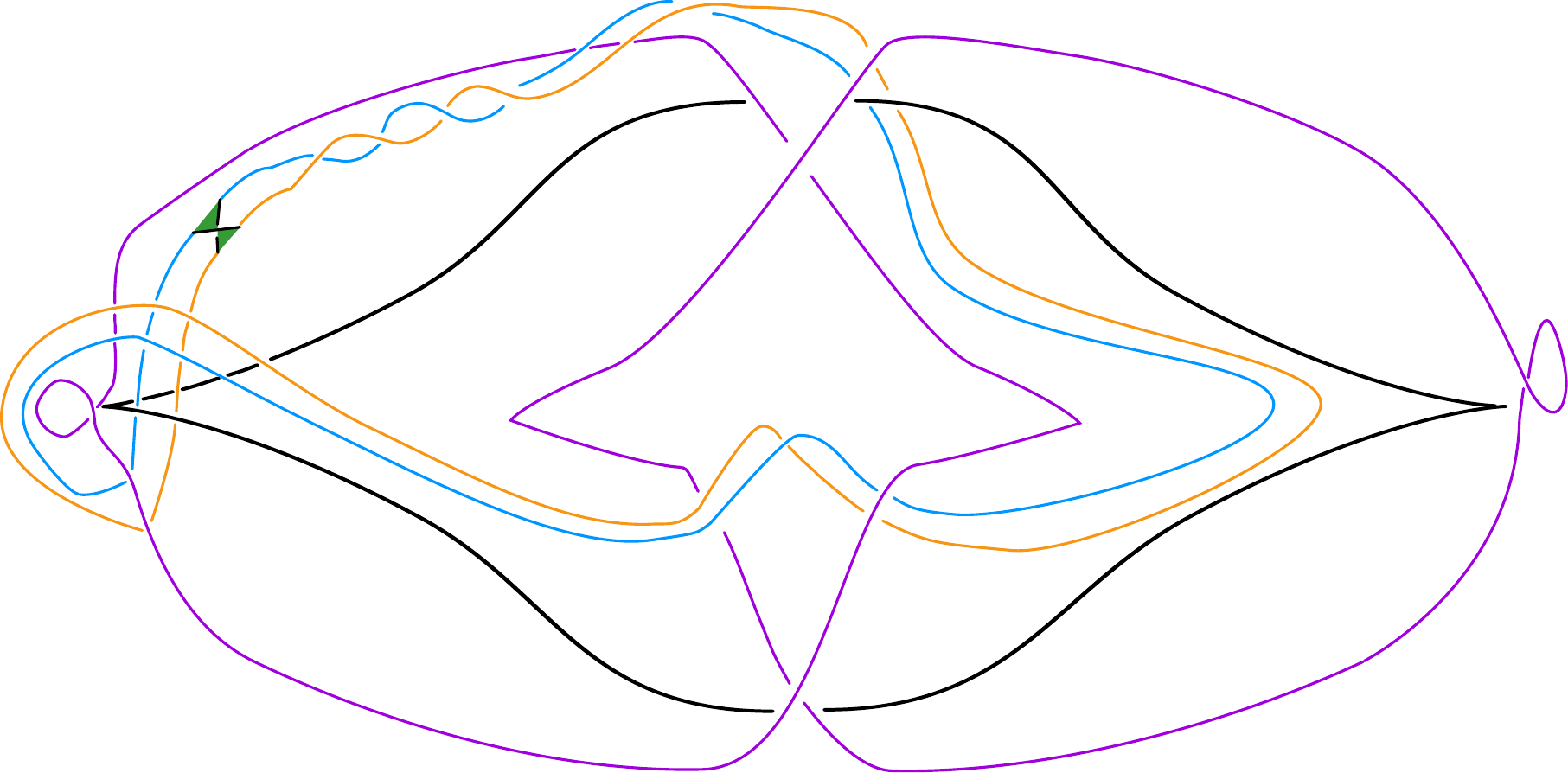}    \caption{A braided, banded unlink diagram for a surface in the cotangent disk bundle to $S^2$.}
    \label{fig:braidedbandedinkwdiagram}
\end{figure}

We then let $H_2=F\times[-\epsilon,\epsilon]$, which contains $\Lambda$ in its core, and let $\overline{H}_1=S^3\setminus H_1$. Note that $H_3=H_2[\Lambda]$. We draw $H_2$ from $F$, as in the local models in \cite[Fig.\ 3]{IS-divides}.   We then view the braided banded unlink $(L,\nu)$ in bridge position with respect to the double $\Sigma$ of $F$.  See Figure~\ref{fig:bandedinkwdiagram_2}.

\begin{figure}[h]
    \centering
    \includegraphics[width=.6\linewidth]{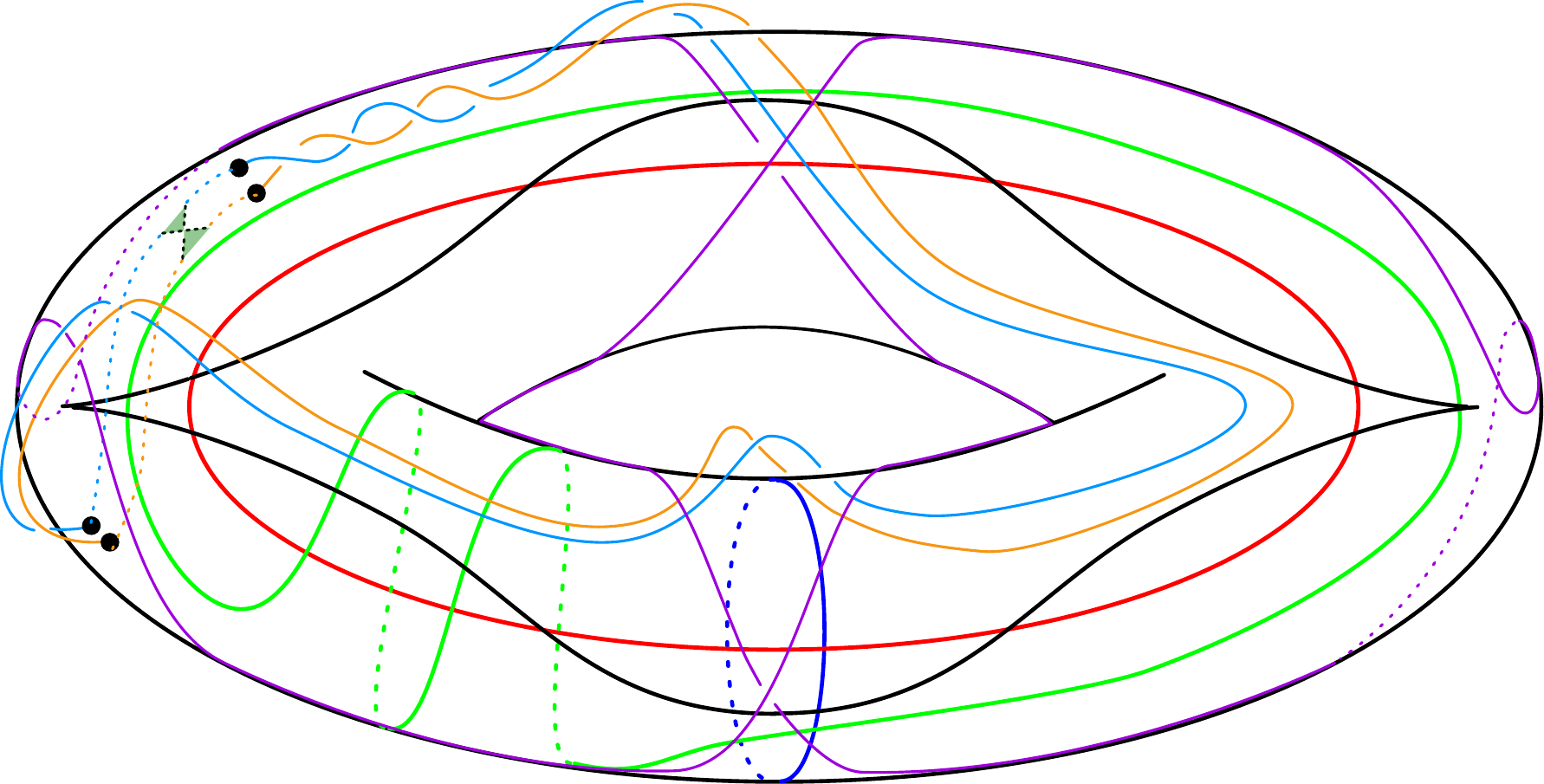}
    \caption{A bridge position banded unlink for the surface in Figure~\ref{fig:bandedinkwdiagram_1}.}
    \label{fig:bandedinkwdiagram_2}
\end{figure}

Finally, we draw the corresponding shadows.  To obtain tangle in $H_3$, recall that we resolve the band on the tangle in $H_2$. The shadows for the tangles in each of the handlebodies are drawn in the bisection diagram in Figure~\ref{fig:shadowdiagram}. 

\begin{figure}[h]
    \centering
    \includegraphics[width=.7\linewidth]{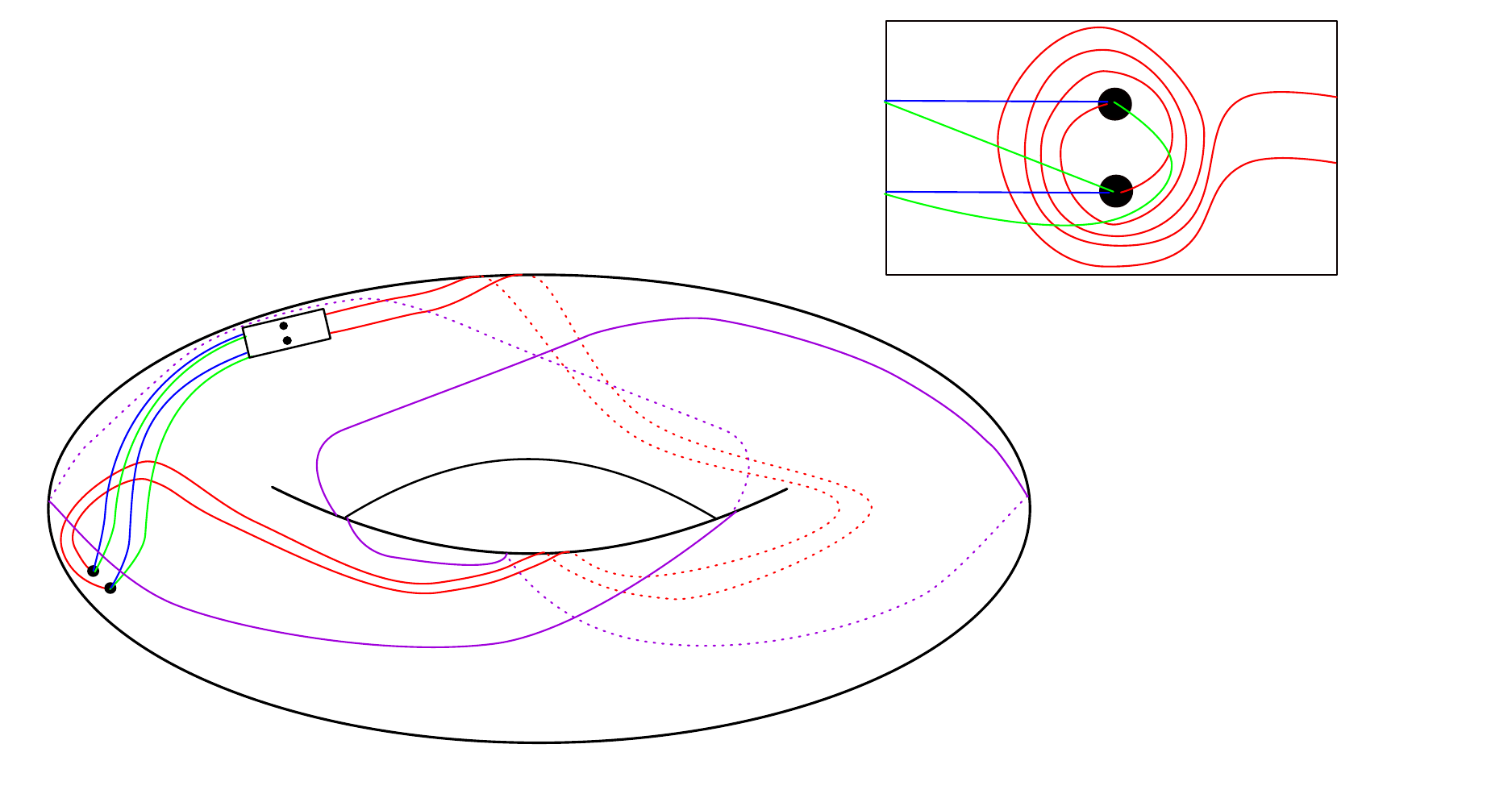}
    \caption{A shadow diagram with divides for the surface in Figure~\ref{fig:bandedinkwdiagram_1}. }
    \label{fig:shadowdiagram}
\end{figure}

The same example can be written in terms of a PA-quasipositive pointed open book. See Figure~\ref{fig:pointedmonodromy}. Let $F$ denote the annulus, with core circle $\gamma$.  Our original 2-component unlink in $S^3$ (in the complement of $\Lambda$) lies in $H_1\cup \overline{H}_2$, which is an open book with page $F$ and monodromy $D_\gamma$.  The surgered link lies in $H_2\cup \overline{H}_3$, which is also an open book with page $F$ and monodromy $D_\gamma$.  Indeed, in Figure~\ref{fig:pointedmonodromy}, the red arcs are carried to blue arcs by the pointed monodromy $\hat\phi_{12}=H_\alpha^3\circ D_{\gamma'}^{-1}\circ D_\gamma^2$, and the blue arcs  are carried to the  green arcs by the pointed monodromy $\hat\phi_{23}=H_\alpha \circ D_{\gamma'}.$

\begin{figure}[h]
    \centering
\includegraphics[width=.7\linewidth]{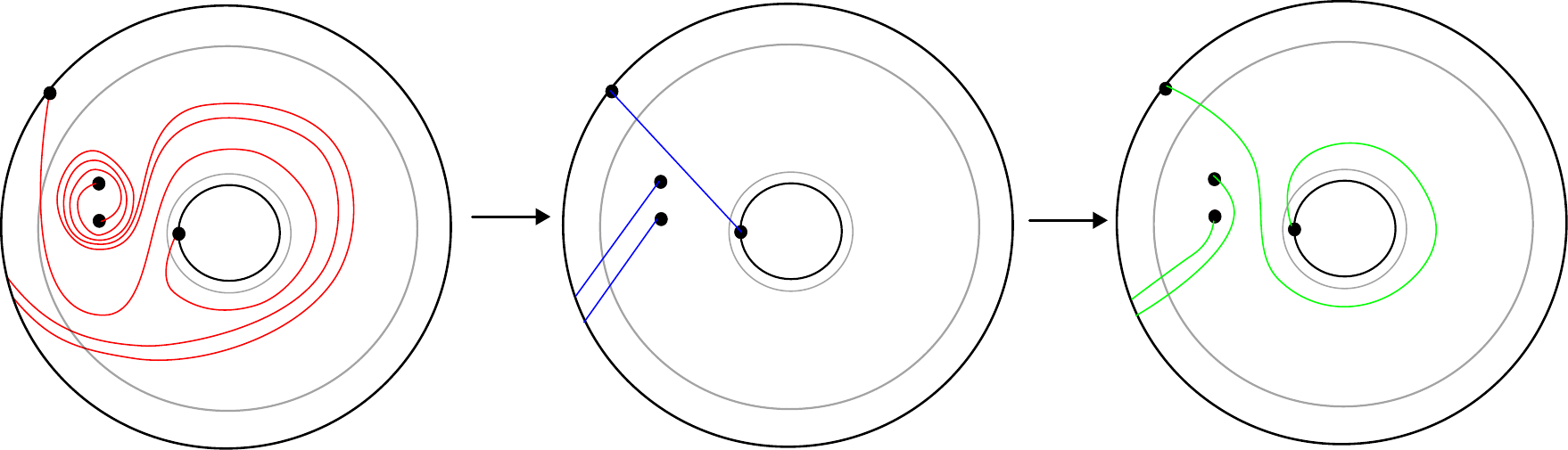}
    \caption{The pointed monodromy for the surface in Figure~\ref{fig:bandedinkwdiagram_1}. }
    \label{fig:pointedmonodromy}
\end{figure}
\end{example}

\section{Proofs of main results}\label{sec:proofs}

In this section, we prove Theorem~\ref{thm: main} and Theorem~\ref{thm: intro_characterization}. 

\subsection{Bridge bisections with divides from ascending surfaces}
We restate Theorem~\ref{thm: main} as Theorem~\ref{thm:existence_bisections}, about the existence of bisections with divides for ascending surfaces. To do that, we show that transverse banded unlinks in banded bridge position induce bridge bisections with divides (Prop.~\ref{prop:banded_gives_bisection}), and every transverse banded unlink can be transversely isotoped into banded bridge position (Prop.~\ref{prop:bandedunlink_implies_bandedbridgeposition}). 

Setting the stage for the propositions, consider a bisection with divides $W=W_1\cup W_2$ as in Remark~\ref{rem:KWtobisection} and an ascending surface $S(U,\nu)$ inside of $W$. Suppose that $(U,\nu)$ is in banded bridge position, suppose that the bands in $\nu$ are inside the handlebody $H_2$, near the boundary of $H_2$. Let $\tau_i=U\cap H_i$ be the braided arcs in the half-open books $H_i$ for $i=1,2$. Recall from the discussion before Lemma~\ref{lem:spine_from_bridgeposition}, that the surgered tangle $\tau_2[\nu]$ can be regarded as a subset of $H_3$. The following lemma shows that $\tau_2\cup \tau_2[\nu]$ is a \textmaxsl unlink. 

\begin{lemma}\label{lem: bridge_position_bands}
Let $(L,\nu)$ be a transverse banded link in $(M,\xi)$ in 
banded bridge position with open book $(\Gamma, \pi)$ and associated contact Heegaard splitting $H_1\cup \overline H_2$. Let $L=\tau_1 \cup \overline{\tau_2}$ be the bridge splitting of $L$ where $\tau_i=L\cap H_i$, $i=1,2$. Suppose the bands $\nu$ lie on the same side as the tangle $\tau_2$. Then, for any Legendrian link $\Lambda$ contained in the Legendrian core of $H_2$, $\tau_2 \cup \overline{\tau_2[\nu]}$ is a transverse unlink with maximal self linking (Def.~\ref{def:unlink}) in $H_2\cup \overline{H_2[\Lambda]}$. 
\end{lemma}

\begin{proof}
Let $H_3=H_2[\Lambda]$ be the surgered handlebody where the surgery coefficient of each component of $\Lambda$ is $tb(\Lambda)-1$. We will show that $\tau_2 \cup \overline{\tau_2[\nu]}\subset H_2\cup \overline{H_3}$ is the result of $|\nu|$ positive stabilizations of an unlink whose components are meridians of the binding. 

Push $\nu$ into a small neighborhood of the page $F_{1/2}$. This way, away from $F_{1/2}$, the braids $\tau_2$ and $\overline{\tau_2[\nu]}$ are mirror images of one another. In particular, $\tau_2 \cup \overline{\tau_2}$ is a transverse unlink with a maximal self-linking number and each component $t$ of $\tau_2 \cup \overline{\tau_2}$ is a 1-braid bounding a disk $E_t$ that intersects the binding $\partial F_0$ exactly once. Observe that, by \ref{item:bb4}, we can choose the disks $E_t$ to have interiors disjoint from the bands in $\nu$. In what follows, we will show that each component of $\nu$ corresponds to a positive destabilization of $\tau_2 \cup \overline{\tau_2}[\nu]$. 

Now, from \ref{item:bb2} and \ref{item:bb3}, the cores of $\nu$ can be projected onto $F_{1/2}$, and their union forms an embedded forest with vertex set equal to the punctures $\tau_2\cap F_{1/2}$. Let $v_1\in \nu$ be a band whose core is a leaf edge of such a tree, and denote by $p$ a puncture that is one of its leaf endpoints. 
Let $t$ be the component of $\tau_2 \cup \overline{\tau_2}$ containing $p$ and $E_t$ its corresponding disk. Notice that $v_1$ is the only band of $\nu$ intersecting $E_t$ in its boundary. In particular, $E_t\cup v_1$ induces a destabilization for $\tau'_2\cup \overline{\tau'_2[\nu]}$ as in Figure \ref{fig:stab}. The $1/2$-framing of $v_1$ implies that this is a positive destabilization. Notice that after performing the destabilization, the remaining bands in $\nu-\{v_1\}$ still satisfy \ref{item:bb2}-\ref{item:bb4} for the braid $\left(\tau_2\cup \overline{\tau_2}\right)-t$. 
We can repeat this process until $\nu$ is empty. The result follows. 
\end{proof}

\begin{figure}[h]
\centering
\includegraphics[width=8cm]{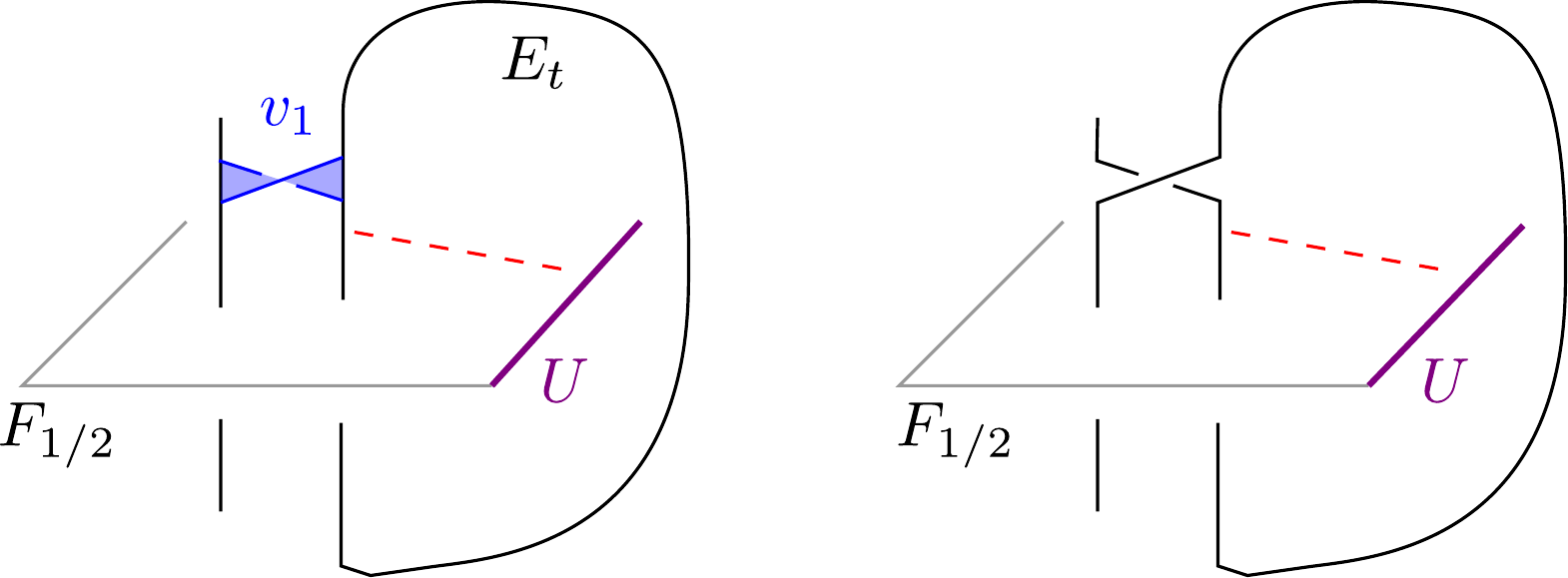}
\caption{Finding a destabilization disk for $L[\nu]$. (left) The disk $E_t$ bounded by a 1-braid $t$, intersecting only one braided band $v_1$. The dotted arc is part of the shadow for $\tau_2$ inside $F_{1/2}$. (right) The same neighborhood drawn after band surgery; the link is now $\tau_2\cup \overline{\tau_2[\nu]}$.}
\label{fig:stab}
\end{figure}

\begin{proposition}\label{prop:banded_gives_bisection}
Let $W$ be a Weinstein domain equipped with a bisection with divides $W=W_1\cup W_2$ induced from a Kirby--Weinstein diagram $(\Lambda,k)$. Suppose that $S(U,\nu)$ is an ascending surface represented by a transverse banded unlink $(U,\nu)$. If $(U,\nu)$ is in banded bridge position with respect to an open book for $(\partial W_1,\xi)$, then $(U,\nu)$ induces a bridge multisection with divides for $S(U,\nu)$. 
\end{proposition}

\begin{proof}
From Section~\ref{rem:KWtobisection}, the bisection with divides $W=W_1\cup W_2$ satisfies $\partial W_1 = H_1 \cup \overline{H_2}$ and $\partial W_2 = H_2 \cup \overline{H_3}$. By Proposition~\ref{realizing-symplectic-surface.prop}, the surface $S(U,\nu)$ is the union of trivial symplectic disks in $W_1$ with boundary $U$, and a cobordism $S_\nu$ between $U$ and $U[\nu]$. Since $U\cap H_2=\tau_2$ and $\nu \subset H_2$, $S(U,\nu)\cap W_2$ is the cobordism between the tangles $\tau_2$ and $\tau_3=\tau_2[\nu]$ induced by $\nu$. Thus, $S(U,\nu)\cap W_2$ is a boundary parallel surface in $W_2$. Moreover, the acyclic condition in Definition~\ref{def:banded bridge position} 
ensures that $S(U,\nu)\cap W_2$ is a collection of disks in $W_2$. 
Now, by Lemma~\ref{lem: bridge_position_bands} 
and since $(U,\nu)$ are in banded bridge position, the braids $\tau_1\cup \overline{\tau_2}$ and $\tau_2\cup \overline{\tau_2[\nu]}$ are \textmaxsl unlinks in $H_1\cup \overline{H_2}$ and $H_2\cup \overline{H_3}$, respectively. In conclusion, $S\cap W_1$ and $S\cap W_2$ are trivial symplectic disk systems as in Definition~\ref{bridge_multisection_divides.def}, and they form a bridge bisection with divides for $S(U,\nu)$.
\end{proof}

\begin{proposition}\label{prop:bandedunlink_implies_bandedbridgeposition}
Let $(U,\nu)$ be a transverse banded unlink in $(M,\xi)$ with open book $(\Gamma, \pi)$. There exists a transverse isotopy of $(U,\nu)$ such that $(U,\nu)$ is in 
banded bridge position with respect to $(\Gamma,\pi).$
\end{proposition}
\begin{proof}
\ref{item:bb1} and \ref{item:bb2} follow from \cite{pavelescu2012braiding} and \cite{hayden21}, respectively. 
Namely, we first transversely isotope $(U,\nu)$ to $(U_0,\nu)$ such that $U_0$ is a quasipositive
braid in $(\Gamma,\pi)$, and such that the bands $\nu$ are fixed during the isotopy, using \cite{pavelescu2012braiding} and \cite{hayden21}. By the definition of band surgery, $U$ and $U[\nu]$ cobound an embedded surface $S$ in $M$. The positivity condition on $\nu$ implies that the characteristic foliation of $S$ has only positive critical (hyperbolic) points. By Remark~\ref{bandsurface.rem}, the characteristic foliation on $S$ is weakly gradient-like with respect to some Morse function on $S$. By Theorem 1.3 of \cite{hayden21}, $U[\nu]$ is also transversely isotopic to a (quasipositive) braid $U_1$ in $(\Gamma,\pi)$. Examining the proof of Theorem 1.3, case ii.b., of  \cite{hayden21}, after stabilizing the braid $U_0$ with respect to $(\Gamma,\pi)$, each $v\in\nu$ is a (1/2)-framed braided band determined by an arc in a page of $(\Gamma,\pi)$ with distinct endpoints.

To show that any transverse banded unlink can be transversely isotoped into braided bridge position with respect to a given open book $(\Gamma,\pi)$, it remains to show that \ref{item:bb3} and \ref{item:bb4} can be achieved by an additional transverse isotopy. Recall from Definition~\ref{def:banded bridge position} that bands satisfying \ref{item:bb3} and \ref{item:bb4} are called dual bands. The following argument will show that if $\omega$ is a collection of dual bands for a braid $U$ and $v$ is a \mbox{(1/2)-framed} braided band $U$ disjoint from $\omega$, then, after at most two positive stabilizations, the bands $\omega\cup \{v\}$ can be made dual to $U$. The result will follow by induction on the number of bands in $\nu$. 

We adopt the notation in Definition~\ref{def:banded bridge position}: the strands of $U$ in each half-open book by $\tau_i=U\cap H_i$, and $\tau_2'\subset F_0\cup F_{1/2}$ is a set of shadows for $\tau_2$. Let $s_2\subset F_{1/4}$ be the arcs of $\tau_2'\cap F_0$ projected into $F_{1/4}$; $s_2$ is equal to the intersection of a set of bridge disks for $\tau_2$ with the page $F_{1/4}$.
Suppose that $\omega$ and $v$ lie in the interior of $F_{[0,1/4]}$ and $F_{[1/4,3/8]}$, respectively. Denote by $\widetilde{v}$ and $\widetilde{\omega}$ the projections of the cores of $v$ and $\omega$ into $F_{1/4}$. As the bands $v$ and $\omega$ are braided (Def.~\ref{def:banded bridge position}), $\widetilde{\omega}$ and $\widetilde{v}$ are unions of embedded arcs in $F_{1/4}$. As $\omega$ is dual to $U$, the arcs of $\widetilde{\omega}$ have pairwise disjoint interiors disjoint from $s_2$ and may intersect with $\widetilde{v}$.  

We first want find an arc $\alpha$ in $F_{1/4}$ connecting an endpoint of $\wt{v}$ with $\partial F_{1/4}$ such that $\interior(\alpha)\cap (s_2\cup \wt{\omega}\cup \wt{v} )$ is empty. As $\wt{v}$ is an arc in $F_{1/4}$, we can find an arc $\beta$ disjoint from $\wt{v}$ without ensuring the condition $\interior(\beta)\cap (s_2\cup \wt{\omega})=\emptyset$; see the left column of Figure \ref{fig:step1}. We use this arc to positively stabilize $U$ in a neighborhood of $F_{1/4}$ (Def.~\ref{def:braid_stab}).
As shown in the figure, we can choose the new intersection between our braid and $F_{1/4}$ to be any point in the interior of $\beta$. We choose this point so that the subarc of $\beta$ touching $\partial F_{1/4}$ is disjoint from $(s_2\cup\wt\omega)$. We denote such a subarc by $\alpha$. Observe that the band $v$ can be slid through the new braid so that the projections $(s_2\cup\wt{v})$ share an endpoint with $\alpha$ and have interior disjoint from $\alpha$. 

\begin{figure}[h]
\centering
\includegraphics[width=10cm]{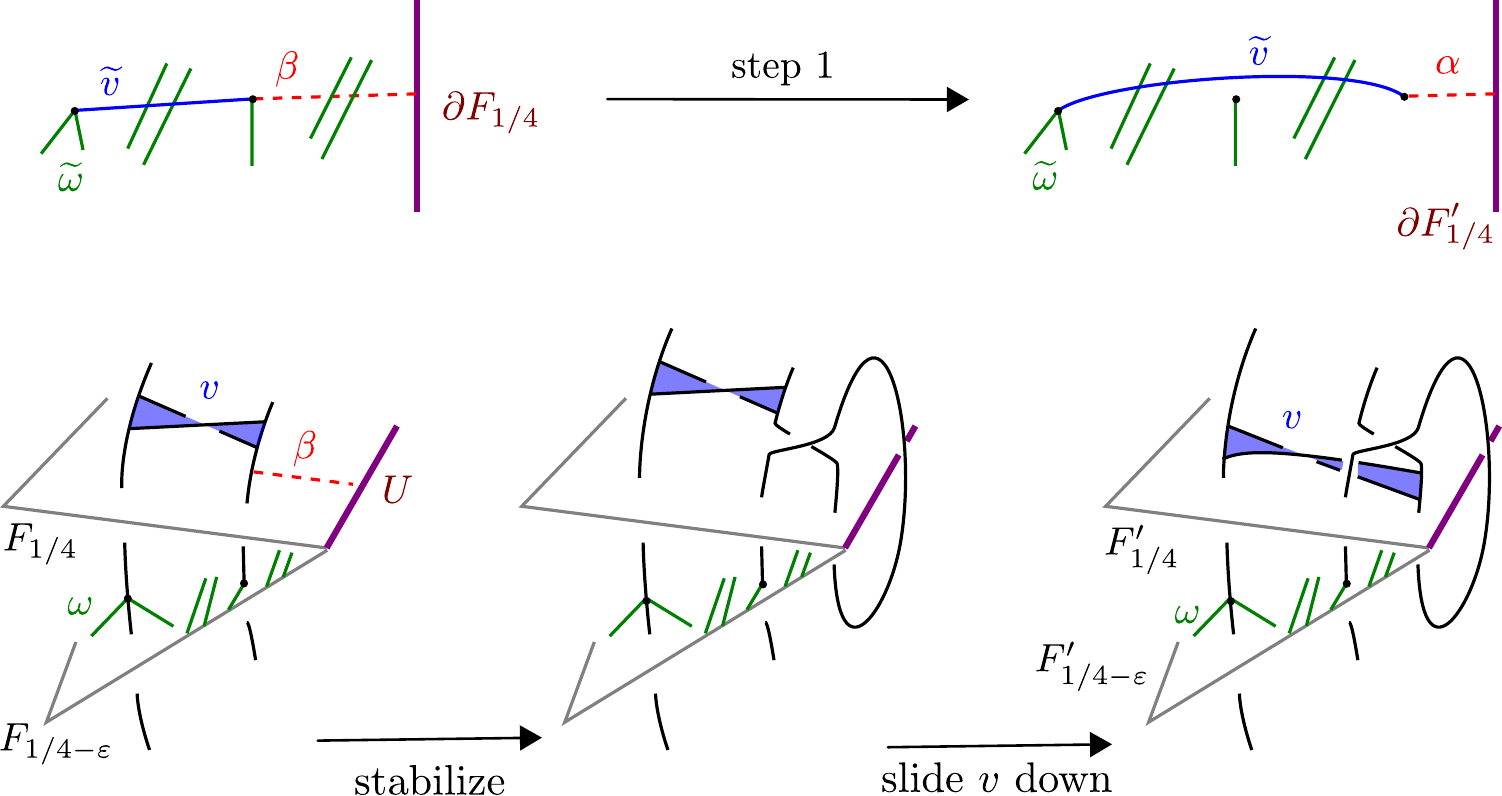}
\caption{In the proof of Proposition \ref{prop:bandedunlink_implies_bandedbridgeposition}, one stabilization suffices to find the arc $\alpha$ with interior disjoint from \mbox{$\wt v \cup s_2\cup\wt \omega$}. The top row depicts the arcs in $s_2$, the cores $\wt \omega$, and $\wt v$ in $F_{1/4}$, while the bottom row depicts a 3-dimensional view of the neighborhood of $F_{1/4}$.
}
\label{fig:step1}
\end{figure}

Let $\alpha\subset F_{1/4}$ be the arc as in the previous paragraph. We use $\alpha$ to guide a positive stabilization of $U$ in a neighborhood of $F_{1/4}$. As shown in Figure \ref{fig:step2}, we can slide $v$ through the new braid so the projection $\wt{v}$ has interior disjoint from $(s_2\cup\wt{\omega})$. 
One can see in the figure that $\wt{v}$ connects $\wt{\omega}$ to a new puncture arising from the stabilization (Def.~\ref{def:braid_stab}). Thus, the cores of $\omega$ and $v$ form a tree, and $\omega\cup\{v\}$ satisfies \ref{item:bb3}. The new tangle $\tau_2$ is the bottom half of the braid in the bottom rightmost panel of Figure~\ref{fig:step2}. We observe that the shadows of the new $\tau_2$ are equal to the old shadows $\tau_2'$, together with one new arc parallel to $\alpha$ (shown in red in the top right panel of Figure~\ref{fig:step2}). In particular, the cores of $\wt{\omega}$ and $\wt{v}$ have disjoint interiors with such shadows; and \ref{item:bb4} is satisfied for $\omega\cup\{v\}$. This finishes the proof.
\end{proof}

\begin{figure}[h]
\centering
\includegraphics[width=10cm]{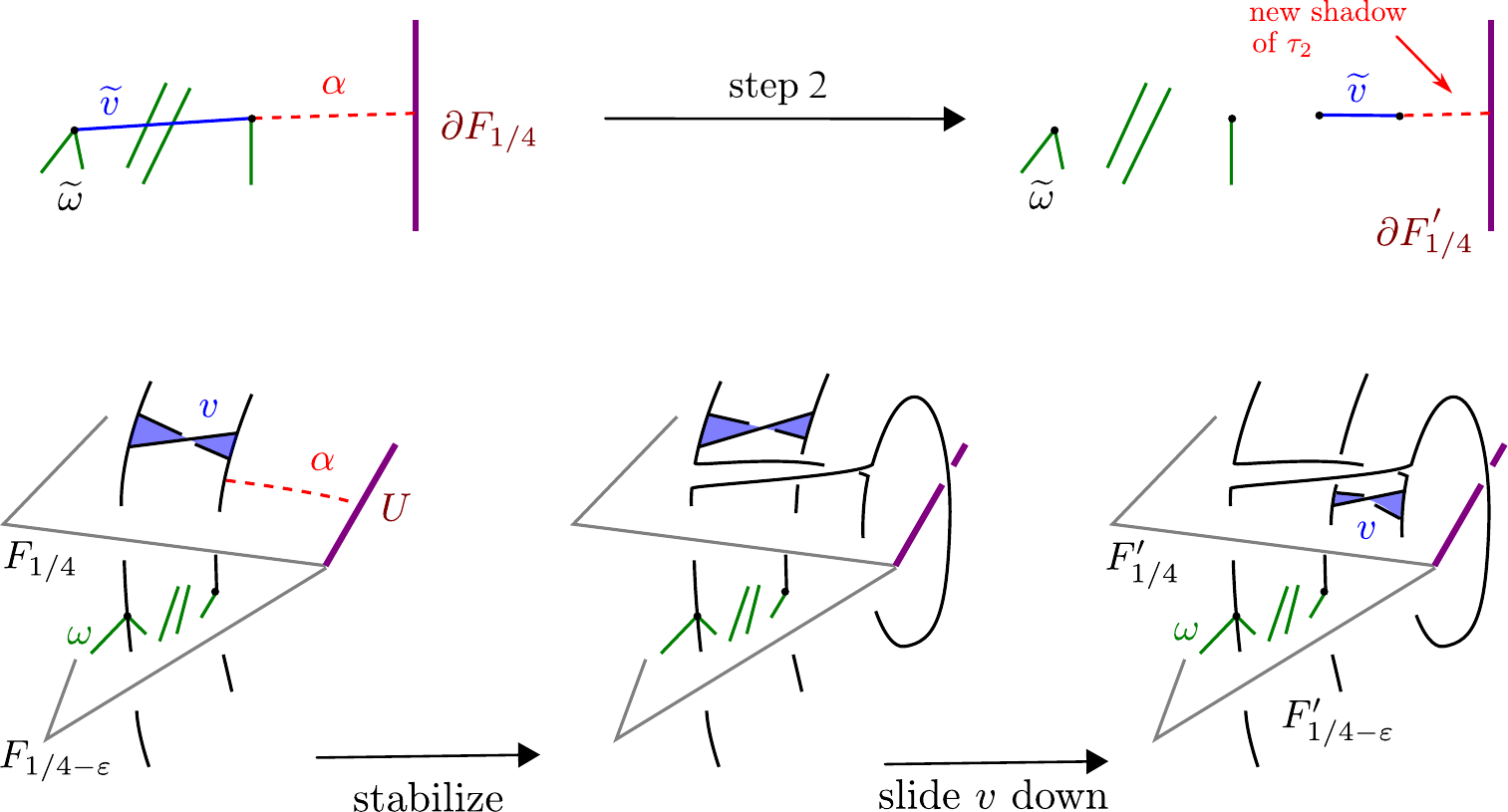}
\caption{How to use the arc $\alpha$ to stabilize $L$ in the proof of Proposition~\ref{prop:bandedunlink_implies_bandedbridgeposition}.} 
\label{fig:step2}
\end{figure}

We are ready to prove Theorem~\ref{thm: main};  Theorem~\ref{thm:existence_bisections} is an alternative and more detailed version of this main result. 

\begin{theorem}\label{thm:existence_bisections}
Let $W$ be a Weinstein domain equipped with a specific handle structure, i.e., a choice of plurisubharmonic function $\rho: W \to \R$. For any bisection with divides $\Tcal$ obtained from $\rho$, any ascending surface $F$ in $W$ can be symplectically isotoped into a bridge position with divides with respect to $\Tcal$. 
\end{theorem}
\begin{proof}
Let $(\Lambda,k)$ be a Kirby--Weinstein diagram for $W$ inducing the bisection with divides $\Tcal$; i.e., $\Tcal=\Tcal\left(\Lambda,k,\Sigma\right)$ (see Sec.~\ref{rem:KWtobisection}). Proposition~\ref{prop: construct_pos_asc} implies that $S$ can be described with a transverse banded unlink diagram $(U,\nu)$ inside $(\Lambda,k)$. The result now follows from Propositions~\ref{prop:bandedunlink_implies_bandedbridgeposition} and~\ref{prop:banded_gives_bisection}.
\end{proof}


\subsection{Ascending surfaces from bridge bisections with divides} \label{sec:stabilizations}

We have seen that the spine of a bridge bisection with divides determines a symplectic surface (Prop.~\ref{prop: spine_determines_symplectic surface}). In this section, we will show that such a surface can be made positive ascending after sufficiently many local modifications. 

We start by defining a modification of a symplectic surface near the boundary. Let $S$ be a properly embedded symplectic surface in a Weinstein domain $(W,\omega)$, and let its boundary $K$ be represented as a braid in any open book supporting the induced contact structure on $\partial W$. 

\begin{definition}\label{boundary_stab.def}
The \emph{boundary stabilization} of $S$ is defined to be the symplectic surface $S'$ obtained by concatenating $S$ and the trace of a positive stabilization of the braid representative of $K$. 
\end{definition}
The surfaces $S'$ and $S$ are isotopic through symplectic surfaces, such that the isotopy restricts to a transverse isotopy on the boundary; see \cite[Fig.\ 1]{Orevkov03} for details. 

\begin{figure}[h]
    \centering
    \includegraphics[width=0.65\linewidth]{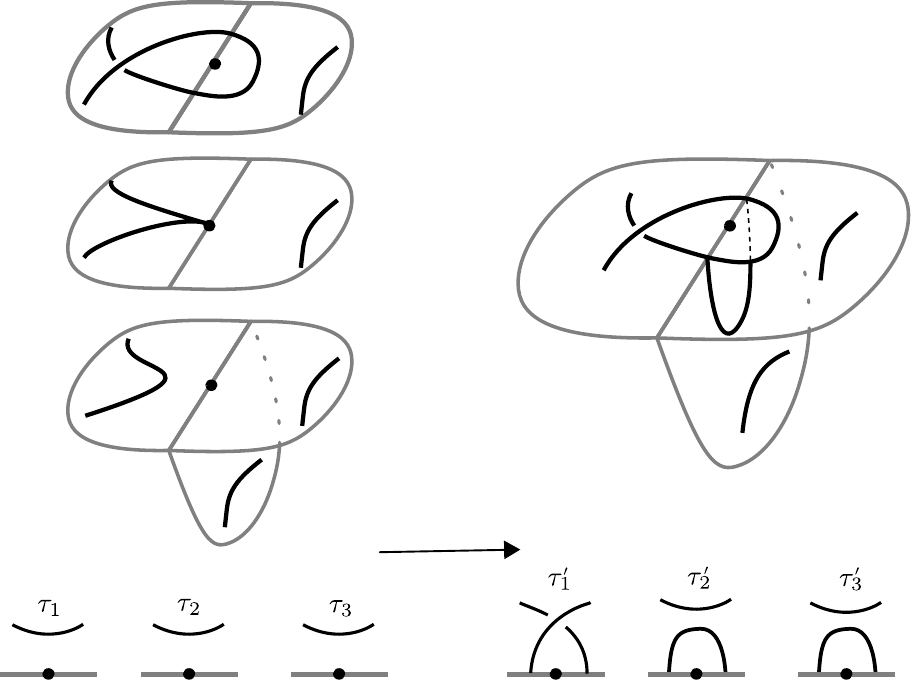}
    \caption{Bisecting a boundary stabilization}
    \label{fig:boundary_stab}
\end{figure}

The following move captures how the spine of a bridge bisection with divides changes under certain boundary stabilizations. This move, called a \emph{positive braided perturbation}, is part of a collection of moves on bridge bisected surfaces that fix the underlying surface~\cite[Sec.\ 3.1]{aranda25}.

\begin{definition}
Let $\Tcal=(H_1,H_2,H_3;\tau_1, \tau_2,\tau_3)$ be the spine of a $b$-bridge bisection with divides and consider $\hat\phi_{i,i+1}$ be the pointed mapping class group of the tangle $\tau_i$. Let $\tau'_1$, $\tau'_2$, $\tau'_3$ be ($b+1$)-stranded braided tangles corresponding to the pointed monodromies:
\[\hat\phi'_{12}=ATB, ~ \hat\phi'_{23}=\hat\phi_{23}, \text{ and } \hat\phi'_{31}=\hat\phi_{31}, \text{ or }\]
\[\hat\phi'_{12}=\hat\phi_{12}, ~ \hat\phi'_{23}=\hat\phi_{23}, \text{ and } \hat\phi'_{31}=AT^{-1}B,\]
where $A,B$ are pointed mapping class group elements, $\hat\phi_{12}=AB$ or $\hat\phi_{31}=AB$, and $T$ is a positive half-twist involving the $(b+1)$-st puncture of the page. The resulting tuple of braided tangles $\Tcal'=(H_1,H_2,H_3;\tau'_1,\tau'_2,\tau'_3)$ is called a \emph{positive braided perturbation} of $\Tcal$.
\end{definition} 

Positive braided perturbations can be seen as a tangle modification by adding a half-twisted band to $\tau_1$ (or $\tau_3$) along an arc in a page of the half-open book, where each of the other two tangles gains one braided strand supporting no monodromy.

\begin{lemma}\label{lem:braid_stab_to_surface_stab}
Let $\Tcal=(H_1,H_2,H_3;\tau_1,\tau_2,\tau_3)$ be the spine of a bridge bisection with divides for $(W,S)$. Suppose that $\Tcal'=(H_1,H_2,H_3;\tau'_1,\tau'_2,\tau'_3)$ be the result of performing a positive braided perturbation on $\Tcal$. Then the symplectic surface described by $\Tcal'$ is equal to a boundary stabilization of $S$. 
\end{lemma}
\begin{proof}
We present the proof for the first type of positive braided perturbation, where $\hat\phi_{12}=ATB$; the second is analogous. Using the notation in Definition \ref{bridge_multisection_divides.def}, the half-open book structures in $H_1$ and $H_3$ determine an open book decomposition in $\partial W=H_1\cup \overline{H}_3$, in which the boundary of $S$, $K=\tau_1\cup \overline{\tau_3}$, is braided. We want to perform boundary stabilization on $S$ and ensure that the resulting surface still comes equipped with a bridge-bisection with divides. To do this, we choose the trace of the positive stabilization of $K=\tau_1\cup \overline{\tau}_3$ to fix $\tau_3$ set-wise as in Figure \ref{fig:boundary_stab}. 

At the 4-manifold level, adding a collar of $\partial W$ to $W=W_1\cup W_2$ gives the same bisection with divides, as we just add a collar of $H_1$ (resp.\ $H_3$) to $W_1$ (resp.\ $W_2$). At the surface level, $S'\cap W'_1$ is equal to a boundary stabilization of $S\cap W_1$, and $S'\cap W'_2$ is equal to $S\cap W_2$ with a small trivial disk added; see Figure \ref{fig:boundary_stab}. The new surface $S'$ intersects the 3-dimensional handlebodies $(H'_1,H'_2,H'_3)$ in the tangles $(\tau'_1,\tau'_2,\tau'_3)$, as desired. 
\end{proof}

\begin{lemma}\label{lem:bridge_to_banded}
Let $\Tcal=(H_1,H_2,H_3;\tau_1,\tau_2,\tau_3)$ be the spine of a $b$-bridge bisection with divides. Then, after performing sufficiently many positive braided perturbations as in Definition 5.6, the resulting spine $\Tcal'=(H_1,H_2,H_3;\tau'_1,\tau'_2,\tau'_3)$ satisfies the following conditions:
\begin{enumerate}
\item The braid $\tau'_1\cup \overline{\tau}'_2$ is braid isotopic to a link of meridians of the binding of the open book in $H_1\cup \overline{H}_2$, together with some positive Markov stabilizations, and
\item there exists a collection of braided bands for the tangle $\tau'_2$ such that $\tau'_2[\nu]$ is braid isotopic to $\tau_3'$ rel boundary. 
\end{enumerate}
\end{lemma}

\begin{proof}
Let $L$ be a braid representative of a max-sl unlink in some tight open book decomposition; $L$ is transversely isotopic to a disjoint union of small meridians of the binding of the open book. According to Pavelescu's work, these two descriptions share a common positive Markov stabilization~\cite{pavelescu2012braiding}. In other words, after enough positive Markov stabilizations, the braid $L$ can be braid isotopic to a positive Markov stabilization of a disjoint union of meridians. Examples of such links $L$ are the braids $\tau_1\cup\overline{\tau}_2$ and $\tau_2\cup \overline{\tau}_3$ in $H_1\cup \overline{H}_2$ and $H_2\cup \overline{H}_3$, respectively. Now, note that positive braided perturbations modify the braids $\tau_1\cup \overline{\tau}_{2}$ and $\tau_2\cup \overline{\tau}_3$ by a positive Markov stabilization or by adding a small meridian of the binding of the corresponding open book decomposition. Hence, there is a tuple $\Tcal'=(H_1,H_2,H_3;\tau'_1,\tau'_2,\tau'_3)$, obtained by some number of positive braided perturbations of $\Tcal$, satistying the following condition for $i=1,2$: $\tau_i\cup \overline{\tau}_{i+1}$ is braided isotopic to positive stabilizations of meridians of the open book of $H_i\cup \overline{H}_{i+1}$. Thus, the first condition is satisfied. To check the second condition, we invoke Lemma 5.3 of \cite{aranda25}, which provides the existence of the braided bands taking $\tau'_2$ to $\tau'_3$. Although the result in 
\cite{aranda25} is only stated for braided diagrams in an open book with a disk page, the proof works for arbitrary pages. 
\end{proof}

We are ready to prove Theorem~\ref{thm: intro_characterization}. Theorem~\ref{restated_thm: intro_characterization} is an alternative and more detailed version of this second main result. 

\begin{theorem}\label{restated_thm: intro_characterization}
Let $(W,S)=(W_1,\mathcal{D}_1)  \cup(W_2,\mathcal{D}_2)$ be a surface in a Weinstein domain equipped with a bridge bisection-with-divides.  Then after performing sufficiently many boundary stabilizations as in Definition~\ref{boundary_stab.def}, the pair $(W,S)$ is symplectically isotopic to a pair $(W,S')$ which admits a bridge bisection-with-divides $(W,S')=(W_1,\mathcal{D}_1')\cup (W_2,\mathcal{D}_2')$, such that $S'$ is positive ascending in $W$.
\end{theorem} 

\begin{proof}
    From Lemmas~\ref{lem:braid_stab_to_surface_stab} and \ref{lem:bridge_to_banded}, it follows that given a pair $(W,S)$ with a bridge bisection-with-divides decomposition, there exists a symplectic isotopy from $S$ to $S'$, and braided perturbations of the bisection-with-divides, such that $S'$ can be described by a transverse banded unlink diagram, which is in banded bridge position with respect to the resulting bisection-with-divides decomposition of $W$. We now note that describing a surface by a banded unlink, which is in banded bridge position, is equivalent to describing it as a quasipositive braid word with respect to a positive allowable open book. Then, appealing to the proof of \cite[Thm.~1.2]{hayden21}, it follows that $S'$ is a positive ascending surface.
\end{proof}

\section{Branched covers}\label{sec: palfs_branched_covers}
Multisections of smooth 4-manifolds are naturally compatible with branched covers: a multisection lifts to a multisection under a branched covering map \cite[Thm.\ 4.1]{lambert2021symplectic}.  We show in Theorem~\ref{branched_covers_lift.thm} that the same is true for multisections with divides. Given a multisection with divides $(W,S)$, and a branched cover $f:(W',S')\rightarrow (W,S)$, we call the resulting bridge multisection with divides of $(W',S')$ a {\it branched cover} of the bridge multisection with divides $(W,S).$  As an application, we show how to equip a smooth branched cover of a Weinstein filling with a Weinstein filling, and explain how to construct bisection with divides for fillings of Brieskorn spheres $\Sigma(2,3,n).$

Compatibility of branched covers with contact and symplectic structures was first studied by Geiges and Auroux respectively \cite{geiges2008introduction, auroux2000_symplectic}.  Compatibility of branched covers with weak fillings was studied by Gironella \cite{gironella2020some}, and that work motivates the definition below.

\begin{definition}[{Following \cite[Thm.\ B]{gironella2020some}}]
    Let $(W',\omega')$ be a (weak, strong, Weinstein) filling of $(M',\xi')$, and $(W,\omega)$ be a (weak, strong, Weinstein) filling of $(M,\xi)$. A {\it branched cover of (weak, strong, Weinstein) fillings} $f:(W',S')\rightarrow (W,S)$ is a smooth branched cover with branching set $S$, such that
    \begin{enumerate}
    \item $S$ and $S'$ are symplectic surfaces in $(W,\omega)$ and $(W',\omega')$, respectively;
    \item $\partial S$ and $\partial S'$ are transverse links in $(M,\xi)$ and $(M',\xi')$, respectively;
    \item $(M',\xi')$ is a contact branched cover of $(M,\xi)$ along $\partial S$ in the sense of \cite[Sec.\ 4.4.1]{geiges2008introduction}; and
    \item $\omega'$ is deformation equivalent 
    to the form $\hat{\omega}_\epsilon$ in \cite[Clm.\ 2.9]{gironella2020some} for $\epsilon$ sufficiently small.
    \end{enumerate}
\end{definition}

In \cite[Thm.\ B]{gironella2020some}, it is shown that given any smooth branched cover $(W',S')\rightarrow (W,S)$ with $(W,\omega)$ a weak filling of $(M,\xi)$, and branching set $S$ a symplectic surface with boundary a transverse link in $(M,\xi)$, the symplectic form $\omega'$ on $W'$ can be chosen so that $(W',\omega')$ is a weak filling of $(M',\xi')$, i.e.\ $f$ is a branched cover of weak fillings.  

It follows from Theorem~\ref{branched_covers_lift.thm} that when $(W,\omega)$ is a Weinstein filling of $(M,\xi)$, the symplectic form $\omega'$ on $W'$ can be chosen so that $(W',\omega')$ is a Weinstein filling of $(M',\xi')$, i.e., $f$ is a branched cover of Weinstein fillings. 

\begin{theorem}\label{branched_covers_lift.thm}
Let $(W,S)$ be a symplectic surface in the Weinstein domain $(W,\omega)$ equipped with a bridge-multisection-with-divides structure $(W,S)=(W_1,\mathcal{D}_1)\cup \ldots \cup (W_n,\mathcal{D}_n)$. Suppose that $S$ is the branching set of a smooth branched cover $f:(W',S')\rightarrow (W,S)$.  Then $(W',S')=(f^{-1}(W_1),f^{-1}(\mathcal{D}_1))\cup \ldots \cup (f^{-1}(W_n),f^{-1}(\mathcal{D}_n))$ is a bridge multisection of $S'=f^{-1}(S)$ in a Weinstein domain $(W',\omega')$, with fillings of each $(f^{-1}(W_i),\omega'|_{f^{-1}(W_i)}
)$ uniquely determined by $f$ and $(W_i,\omega|_{W_i})$. In particular, $f$ is a branched cover of Weinstein fillings.
\end{theorem}

\begin{proof}
Let $f:(W',S')\rightarrow (W,S)$, be a smooth branched cover, where $(W,S)=(W_1,\mathcal{D}_1)\cup \ldots \cup (W_n,\mathcal{D}_n)$ is a bridge multisection of with divides of a symplectic surface $S$.  We show that the pieces of the decomposition $(W',S')=(f^{-1}(W_1),f^{-1}(\mathcal{D}_1))\cup \ldots \cup (f^{-1}(W_n),f^{-1}(\mathcal{D}_n))$ form a bridge multisection with divides.

As $(W,S)=(W_1,\mathcal{D}_1)\cup \ldots \cup (W_n,\mathcal{D}_n)$ is a bridge multisection with divides, we have the following:

\begin{enumerate}
\item $(\Sigma, \{ p_1,\dots, p_{2b}\})=\cap_{i=1}^n (W_i,\mathcal{D}_i)$ is a closed, orientable surface with $2b$ marked points. 
  \item $(H_i,\tau_i)=(W_i,\mathcal{D}_{i})\cap(W_{i+1},\mathcal{D}_{i+1})$ is a $b$-strand braid in the half open book structure on $H_i$, such that consecutive unions $(H_i,\tau_i)\cup (H_{i+1},\tau_{i+1})$ form a braided $\textmaxsl$ unlink in the corresponding open book for the tight contact structure on $H_i\cup\bar{H}_{i+1}$.
    \item $\mathcal{D}_i$ is a trivial symplectic disk system in $(W_i,\omega_i)$, which is a filling of the tight contact structure on $H_i\cup\bar{H}_{i+1}$.     
\end{enumerate}
The above decomposition lifts to a smooth multisection of $(W',S')=(f^{-1}(W_1),f^{-1}(\mathcal{D}_1))\cup \ldots \cup (f^{-1}(W_n),f^{-1}(\mathcal{D}_n))$ as proven in \cite[Thm.\ 4.1]{lambert2021symplectic} for smooth trisections. Therefore, it remains to check the contact and symplectic data lift as required.

The contact structure $\xi_i$ on $H_i\cup H_{i+1}$ determines a contact structure $\xi_i'$ on $f^{-1}(H_i)\cup f^{-1}(H_{i+1})$, the branched cover of $H_i\cup H_{i+1}$ along the transverse link $\tau_i\cup \overline{\tau}_{i+1}$, that is unique up to isotopy \cite[Prop.\ A]{gironella2020some}, \cite{geiges2008introduction}. We first show that since $\xi_i$ is supported by the open book structure corresponding to the contact Heegaard splitting $H_i\cup H_{i+1}$, it follows that $\xi_i'$ is supported by the open book structure corresponding to the contact Heegaard splitting $f^{-1}(H_i)\cup f^{-1}(H_{i+1})$.  This is stated in \cite[Thm.\ 3.4.1]{casey2013branched} for the case where the base is $(S^3,\xi_{std})$. 

 As the contact handlebodies $(H_i,\xi_i)$ all induce the same dividing set $\Gamma$ on $\Sigma$, the contact handlebodies $(f^{-1}(H_i),\xi_i')$ induce the same dividing set $f^{-1}(\Gamma)$ on $f^{-1}(\Sigma)$. By \cite[Thm.\ B]{gironella2020some}, there is a symplectic structure $\omega_i'$ on $f^{-1}(W_i)$ such that $(f^{-1}(W_i),\omega_i')$ is a weak filling of $(f^{-1}(H_i)\cup f^{-1}(H_{i+1}),\xi_i')$. As $(f^{-1}(H_i)\cup f^{-1}(H_{i+1}),\xi_i')$ is contactomorphic to $(\#^{k_i}S^1\times S^2,\xi_{std})$, this weak filling is unique up to symplectic deformation, so it is Weinstein \cite{mcduff_rational-ruled, Wendl_strongly} 
 Also by the construction in \cite[Thm.\ B]{gironella2020some}, the trivial disk systems $f^{-1}(\mathcal{D}_i)$ are symplectic with respect to $\omega_i'.$ As a result, the boundary of each $\mathcal{D}_i$ is a \textmaxsl transverse unlink in $(f^{-1}(H_i)\cup f^{-1}(H_{i+1}),\xi_i')$. Therefore, we have a bridge-multisection-with-divides structure on $(W',S')$ and $f$ is a branched cover of Weinstein fillings.
\end{proof}

\begin{remark}
For 1-handlebodies $\#^k S^1 \times B^3$, there exist non-isotopic but deformation equivalent forms \cite{wang2024note}.
Thus Theorem~\ref{branched_covers_lift.thm} does not specify an isotopy class, but rather a deformation equivalence class of symplectic forms.
This is consistent with the fact that the spine of a multisection-with-divides determines the form up to deformation; see \cite[Rem.~1.2]{IS-divides}.
\end{remark}

In \cite{loi2001compact}, Loi and Piergallini show that every Weinstein domain is a branched cover of $B^4$ along a positive ascending surface. Combining their result with Theorem~\ref{thm: main} and Theorem~\ref{branched_covers_lift.thm}, we obtain a new proof of the following theorem from \cite{IS-divides}.

\begin{corollary}[{\cite[Thm.\ 3.1]{IS-divides}}]\label{Cor: existence_multisections_Weinstein}
    Every compact Weinstein domain admits a bisection with divides.
\end{corollary}
\begin{proof} 
Let $(W,\omega)$ be a Weinstein domain. Then by \cite{loi2001compact}, there is a smooth branched cover $f:(W, S')\rightarrow (B^4,S)$ with branching set $S$ a positive, ascending surface in $(B^4,\omega_{std})$. By Theorem~\ref{thm: main}, $S$ can be isotoped through positive ascending surfaces, so that it is in bridge position with respect to the genus 0 bisection with divides of $(B^4,\omega_{std})$. Taking the branched cover of this bisection yields a bisection with divides structure on $(W,\omega)$ by Theorem~\ref{branched_covers_lift.thm}.    
\end{proof}

\begin{corollary} Every compact Weinstein domain $(W,\omega)$ can be represented by a positive braided tri-plane diagram labeled with transpositions in the symmetric group $S_n$. Conversely, every such braided tri-plane diagram determines a bisection-with-divides structure on $(W,\omega)$.
\end{corollary}
\begin{proof}
    The first statement follows from \cite{loi2001compact} and Theorem~\ref{thm: main}.  The second statement follows from \cite{loi2001compact}, Theorem~\ref{thm: main}, and Theorem~\ref{branched_covers_lift.thm}.
\end{proof}

In \cite{blair2024note}, it is shown that every smooth bisection $X=X_1\cup X_2$ of a 4-manifold with boundary is a three-fold simple cover of the genus-0 bisection of $B^4$, with branching set a properly embedded surface in bridge position with respect to the genus-0 bisection of $B^4$. A positive answer to the following question would show that every bisection-with-divides structure on $(W,\omega)$ can be encoded by a transposition-labeled, positive, braided tri-plane diagram.

\begin{question}
    Let $W=W_1\cup W_2$ be a bisection with divides for a Weinstein domain $(W,\omega)$.  Is there a bridge-bisected surface $S$ with divides in the genus-0 bisection of $(B_4,\omega_{\text{std}})$ such that a branched cover of $(B^4,S)$ is the bisection with divides $W=W_1 \cup W_2$?
\end{question}

\begin{example}
    Consider the bridge-bisected surface bounded by the trefoil in Figure~\ref{fig:fig_qp_trefoil}. We will construct a bisection-with-divides structure on its double branched cover.  The double branched cover of the positive trefoil, with contact topology orientation convention, is the lens space $L(3,2).$  This lens space has a unique tight contact structure $\xi_{std}$ given by the quotient of the standard contact structure on $S^3$ \cite{Honda2000}, and $(L(3,2),\xi_{std})$ has a unique Weinstein filling described by Kirby--Weinstein diagram a \textmaxsl Hopf link, by \cite{lisca2008symplectic, etnyre2021symplectic,christian2023some}.
    
     We construct a bisection diagram for the bisection with divides structure of genus 3 on its double branched cover.  Note that the algorithm in \cite{IS-divides} produces a genus 5 bisection from the standard diagram of the Legendrian Hopf link, so branched covers may be a useful tool for producing bisection-with-divides of low genus.  The shadow diagram of this surface is a collection of arcs on a sphere with 8 marked points, 4 on each side of the dividing set $\Gamma$; see Figure~\ref{fig:trefoil_shadows}. 
    \begin{figure}
        \centering
        \includegraphics[width=0.5\linewidth]{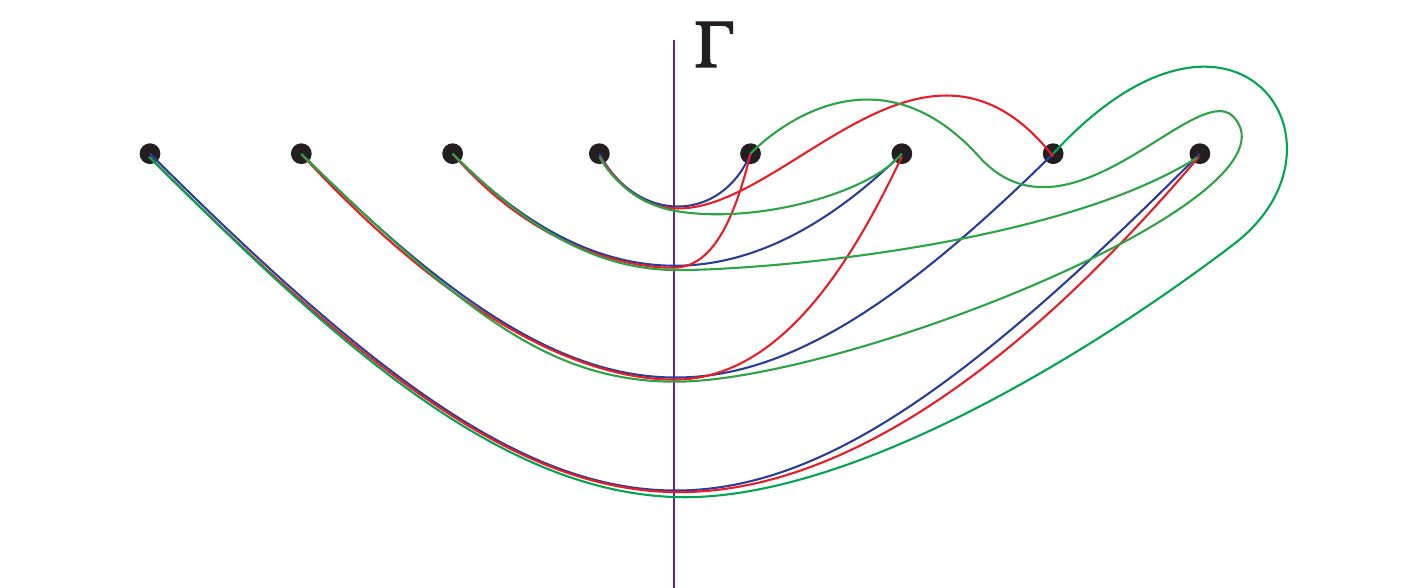}
        \caption{Shadow diagram for the tri-plane diagram of the trefoil surface in Figure~\ref{fig:fig_qp_trefoil}, with dividing set $\Gamma$.}
        \label{fig:trefoil_shadows}
    \end{figure}
    The double branched cover of this sphere along the 8 marked points is a genus 3 surface $\Sigma$, and the dividing set $\Gamma$ lifts to a pair of closed curves separating $\Sigma$ into two copies of a twice-punctured torus $F^\pm$.  The lifts of the shadows pair off to form boundaries of compressing disks in the three handlebodies $H_1$, $H_2$, and $H_3$, resulting in a bisection diagram with divides of genus 3.  See Figure~\ref{fig:trefoil_bisection_cover}. 

    We can also read off the decomposition of the monodromy factorization $\phi=\phi_{2,3}\circ\phi_{1,2}$ from this shadow diagram, namely:
    $$\phi_{1,2}=D_{\gamma_1}\circ D_{\gamma_2}\circ D_{\gamma_3}, \text{ and}$$
    $$\phi_{2,3}=D_{\gamma_1}\circ D_{\gamma_3},$$
    where the curves $\gamma_i$ are as shown in Figure~\ref{fig:trefoil_bisection_cover}.
    \begin{figure}
        \centering
        \includegraphics[width=\textwidth]{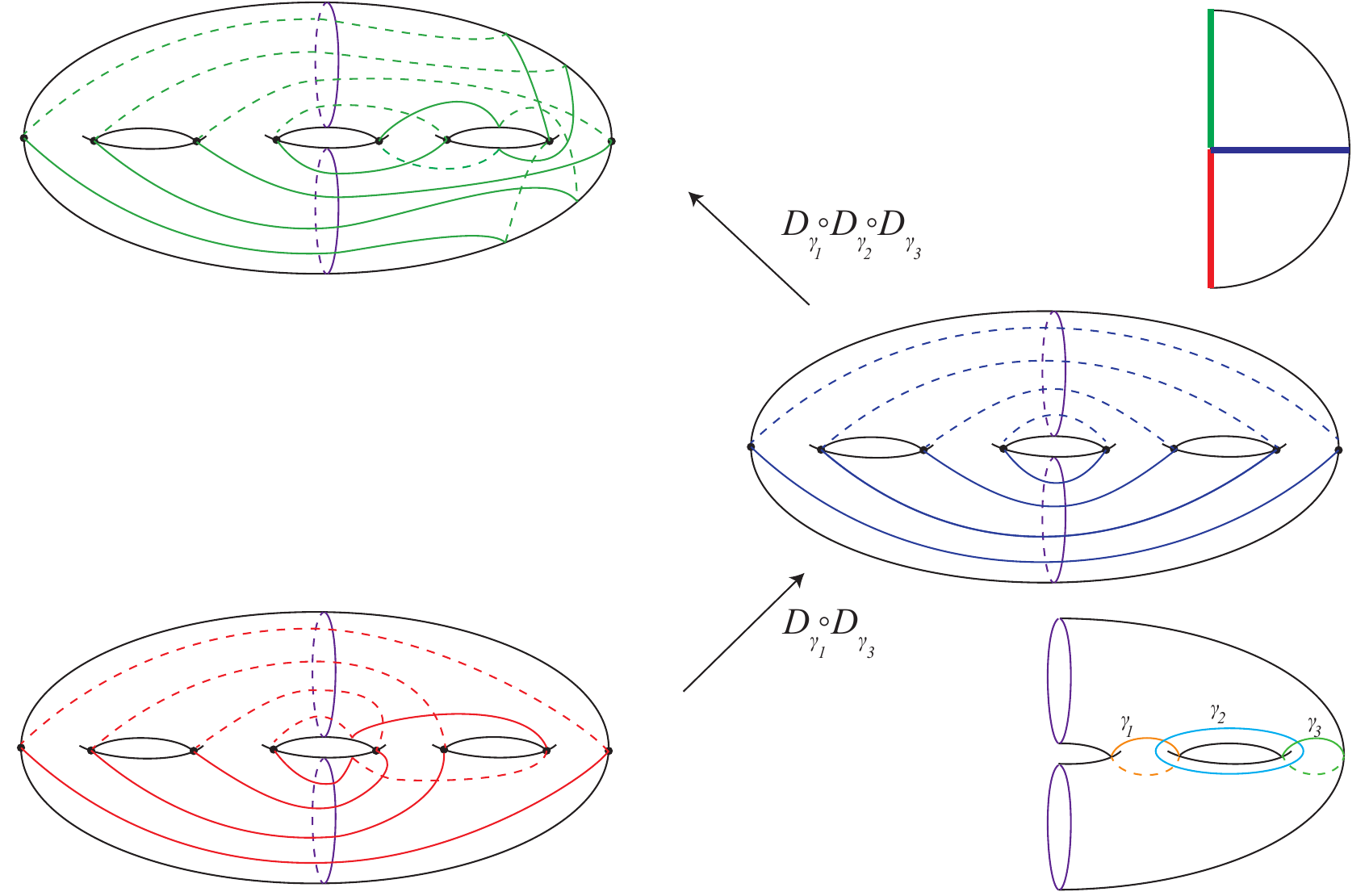}
        \caption{Bisection with divides for the double branched cover of the trefoil surface in Figure~\ref{fig:fig_qp_trefoil}, which has Kirby--Weinstein diagram a Hopf link with \textmaxsl components.}
        \label{fig:trefoil_bisection_cover}
    \end{figure}
   
\end{example}

\begin{example} There is a Weinstein filling of each Brieskorn sphere $\Sigma(2,3,n)$ (the $n$-fold cyclic cover of the trefoil) admitting a genus $3n-3$ 
bisection with divides. 
Therefore, taking the $n$-fold cyclic cover of the bridge-bisection in Figure~\ref{fig:fig_qp_trefoil} produces a bisection with divides for a Weinstein filling of $\Sigma(2,3,n)$, and an Euler characteristic computation shows that the genus of this bisection is $3n-3.$ Kirby--Weinstein diagrams for fillings of these and other Brieskorn spheres are given in \cite{Akhmedov2021ChainSurgeries}. Cyclic branched covers of bisections for minimal-genus Seifert surfaces of torus knots can be used to produce bisections with divides for fillings of $\Sigma(p,q,r).$
\end{example}

\section{Questions and future directions}
\subsection{Complements of multisections in Lefschetz fibrations} Upcoming work \cite{BRW:PANLF} shows that complements of multisections in Lefschetz fibrations admit Weinstein structures. This is a generalization of a classical result which endows the complement of a complex curve in $\mathbb{C}^2$ with a Stein structure \cite{simha}. 
One can expect that placing these positive ascending surfaces in a bridge-bisected position can be used to obtain multisection-with-divides diagrams of the complements. One can ask the following question:

\begin{question}
    Can one use the bridge-bisected diagrams of positive ascending surfaces to obtain `efficient' bisection-with-divides diagrams for the complements of these surfaces?
\end{question}

This approach has been followed in the smooth setting in the following series of works. Using the results of Meier--Zupan \cite{meier2017bridge, meier2018bridge} that show surfaces in 4-manifolds can be placed in bridge position, Gay--Meier and Kim--Miller \cite{gay_meier, kim_miller} showed how to obtain trisection diagrams for the complements of surfaces and then described surgery along 2-knots. This was then applied by Naylor \cite{naylor_twist} to obtain a trisection-theoretic proof of a theorem of Katanaga--Saeki--Teragaito--Yamada \cite{KSTY_gluck-price} that relates Gluck and Price twists of 4-manifolds.

\subsection{$\mathcal{L}$-invariants}
In smooth topology, one can use combinatorial complexes associated to 2-dimensional surfaces to define invariants of trisected 4-dimensional objects. Examples of such invariants are the $\mathcal{L}$-invariants defined for (a) closed 4-manifolds~\cite{L_Invariant2018}, (b) 4-manifolds with boundary~\cite{Rel_L_invariant_2021}, (c) 2-knots in $S^4$~\cite{BCTT22_Distance, APZ_kirbythompson23}, and (d) knotted surfaces in arbitrary 4-manifolds~\cite{aranda23pants}. Recently, the authors of \cite{contact_L_invariant_2025} extended the smooth $\mathcal{L}$-invariants to symplectic 4-manifolds equipped with a multisection with divides. From their work, one should expect these ideas to generalize to bridge multisections with divides. 
\begin{problem}\label{prob:L-inv}
    Define a contact $\mathcal{L}$-invariant for ascending surfaces in Weinstein domains. 
\end{problem}
Roughly, $\mathcal{L}$-invariants measure how far apart the handlebodies in the spine of a trisection are. Thus, equipped with the equivalences in Theorem~\ref{thm:equivalent_descriptions}, it would be interesting to exploit the algebraic structure of the 
mapping class group of a surface to construct a lower bound for this new contact $\mathcal{L}$-invariant from Problem~\ref{prob:L-inv}. 

\bibliographystyle{alpha}
\bibliography{bridge}

\end{document}